\date{2017-01-11}
\title{\bf{Symmetric road interchanges}}
\author{Valentas Kurauskas and Ugn\.{e} \v{S}iurien\.{e}}

  \documentclass[11pt]{article}

  \usepackage[rgb]{xcolor}

  \usepackage{amsmath}
  \usepackage{amssymb}
  \usepackage{latexsym}
  \usepackage{graphicx}

  \usepackage{color} 
  \usepackage{enumerate}
  \usepackage{setspace}
  \usepackage{cite} %sort citations when citing

  \usepackage{hyperref}
  \usepackage{subcaption}

  \hypersetup{ %should work [hidelinks] in the new version, but not yet,
    colorlinks=false,
    pdfborder={0 0 0},
  }
  \newenvironment{proof}{\noindent{\bf Proof\,}}{\hspace*{\fill}$\Box$}
  \newenvironment{proofof}[1]{%
  \noindent {\bf Proof of #1}}%
  {\hspace*{\fill}$\Box$}

\newcommand{\m}[1]{}

  \newtheorem{theorem}{Theorem}[section]
  \newtheorem{lemma} [theorem] {Lemma}%[section]
  \newtheorem{prop} [theorem] {Proposition}%[section]
  \newtheorem{claim} [theorem] {Claim}%[section]

\def\pr{{\mathbb P}\,}

\def\Lgeom{L_{C}^{*}}
\def\Lgeomt{\tilde{L}_{C}^{*}}
\def\Lcomb{L_{C}}

\def\ex{{\rm ex}}
\def\mul{{\rm mul}}
\def\solid{{\Rightarrow}}
\def\dotted{{\rightarrow}}

\def\leq{\leqslant}
\def\geq{\geqslant}
\def\Ga{\mathcal{M}_0}

\def\S{\mathcal{S}}
\def\B{\mathcal{B}}
\def\IG{{\rm Iso}}

\mdseries\itshape

\usepackage{libertine}

\begin{document}

\maketitle

\begin{abstract}
%    In \cite{kurauskas2015} we studied the minimum genus of embeddings of the complete graphs $K_{n,n}$ that correspond
%    to complete road interchanges. Here we consider the minimum genus and geometric realisations of symmetric such embeddings.
A road interchange where $n$ roads meet and in which the drivers are not allowed to change lanes can be modelled as an embedding of a 2-coloured (hence bipartite) multigraph $G$ with equal-sized colour classes into an orientable surface such that there is a face bounded by a Hamiltonian cycle  \cite{kurauskas2015}.
The case of $G$ a complete bipartite graph $K_{n,n}$ corresponds to a complete $n$-way interchange where drivers approaching from each of $n$ directions can exit to any other direction. The genus of the underlying surface can be interpreted as the number of bridges in the interchange.

In this paper we study the minimum genus, or the minimum number of bridges, of %\emph{symmetric} such embeddings. We find the minimum genus for
a complete interchange with a restriction that it is symmetric under the cyclic permutation of its roads. We consider both (a) abstract combinatorial/topological symmetry, and (b) symmetry in the 3-dimensional Euclidean space $\mathbb{R}^3$. The proof of (a) is based on the classic voltage and transition graph constructions. For (b) we use, among other techniques, a simple new combinatorial lower bound. 
\end{abstract}

\bigskip

\section{Introduction}

We call a tuple $I=(G,H,\mathcal{M}, \S)$ a \emph{weaving-free interchange}, or simply \emph{an interchange}, 
if $\mathcal{M}$ is an embedding of a 2-coloured multigraph $G$ into an orientable surface $\S$ such that $G$ has the same number
of black and white vertices and the embedding $\mathcal{M}$ has a face with a directed Hamiltonian cycle $H$ as its boundary walk \cite{kurauskas2015}. 
Here pairs of subsequent white and black vertices on $H$ model the roadways (groups of lanes with traffic flow in the same direction) that enter and exit the interchange respectively. The images under $\mathcal{M}$ of the edges of $G$ model the actual lanes inside the interchange. Finally, the Hamiltonian cycle $H$ corresponds to the perimeter of the interchange and bounds the corresponding outer face.
$I$ is called \emph{$t$-way} if $G$ has $t$ black and $t$ white vertices. $I$ is \emph{complete} if there is an edge $uv$ for each white vertex $u$ and each black vertex $v$ (multiple edges are allowed). 

For even $n$, a construction for complete $n$-way interchanges with minimum genus was described in \cite{kurauskas2015}.
The problem is equivalent to minimizing the genus of an embedding of $K_{n,n}$ such that one face is bounded by a Hamiltonian cycle.
The solution implies one of the special cases of the conjecture on the minimum genus of a complete tripartite graph.
After several decades of partial progress, see \cite{ellinghamstephens2009, esz2006, ksz2004}, the complete proof of the conjecture has been announced very recently by Ellingham, Stephens and Zha \cite{esz2018}.

In this paper we impose an additional symmetry restriction. 
Symmetric embeddings or interchange layouts result by combining a number of identical pieces. This may be an advantage for aesthetical or practical reasons, especially when we have 3-dimensional realisations as in Theorem~\ref{thm.geometric.opt} below. 
%Additionaly we hope that methods used in understanding of the `easier' case of symmetric embeddings can contribute to solving the minimum genus problem of complete tripartite graphs, which is still open in general \cite{ellinghamstephens2009, esz2006, ksz2004}.

Our first result yields the minimum genus of interchanges where the cyclic permutation of roads (shifting vertices along the Hamiltonian cycle by two positions) preserves the rotation system. This kind of combinatorial/topological symmetry is classically modelled by the voltage graph construction; for definitions we follow Gross and Tucker \cite{grosstucker1987}, see also Section~\ref{sec.combinatorial}.
\begin{theorem}\label{thm.topsym}
    Let $n\ge 3$ be an integer.
    Let $\mathcal{M}$ be an embedding of $K_{n,n}$ derived from a loopless embedded voltage graph on two vertices
    and voltage group $\mathbb{Z}_n$. Suppose $\mathcal{M}$ has a face $F_H$ bounded by a Hamiltonian cycle.
    Let $p_1$ be the smallest prime divisor of $n$.
    The genus of $\mathcal{M}$ is at least $\Lcomb(n)$ where
    \small
    \begin{align*}
    \Lcomb(n) =  \begin{cases}
                         \frac{n(n-2)} 4, &\mbox{if $n$ is even}; \\
                         \lfloor \frac{n(n-1)} 4 \rfloor + 1 - \frac 1 2 \left( \frac n {p_1} + p_1\right), &\mbox{if } n \equiv 3\,(\bmod\,4), p_1 \ne n \mbox{ and } p_1^2 \nmid n; \\
                         \lfloor \frac{n(n-1)} 4 \rfloor + 1 - \frac 1 2 \left( \frac n {p_1} + 1\right), &\mbox{if } n \equiv 3\,(\bmod\,4)  \mbox{ and } p_1^2 \mid n; \\
                         \frac{n(n-1)} 4 - 1,   &\mbox{if }n \equiv 1\,(\bmod\,4), 3 \mid n \mbox { and } 9 \nmid n; \\
                         \lfloor \frac{n(n-1)} 4  \rfloor, &\mbox {otherwise}.
                 \end{cases}
    \end{align*}
    \normalsize
    
    Furthermore, this lower bound is best possible, and the derived embedding $\mathcal{M}$ that achieves genus $\Lcomb(n)$
    and $F_H$ 
    can be chosen so that $F_H$ is generated by a face of size 2 in the base embedding.
\end{theorem}
The fact that $\Lcomb(n)$ depends on the prime divisors of $n$ may be seen as a remote link to the well established theory of regular embeddings of complete bipartite graphs $K_{n,n}$ where the
symmetry requirement is much stronger \cite{jones2010}.

%
%The fact that $\Lcomb(n)$ depends on the prime divisors of $n$ suggests a remote connection
%with a well established theory of regular embeddings of complete bipartite graphs $K_{n,n}$, where the
%symmetry requirement is much stronger \cite{jones2010}.
%A slightly surprising fact is that $\Lcomb(n)$ depends on its prime factors. In contrast, the proven
%or conjectured genus of various dense graph classes (complete graphs, complete bipartite graphs, complete tripartite graphs, etc)
%is a quadratic function in $n$ that depends only on $n \bmod \, k$ for some small number $k$.

%It is not always
%possible to preserve the symmetries of an embedding if one wants to 
%realise a surface with an embedded graph in $\mathbb{R}^3$.
%Therefore the next natural question we consider is the minimum genus of an interchange
%with a given symmetry type in $\mathbb{R}^3$.
% In this paper we focus on embeddings that have n-fold rotational symmetry. We will see below that for 3-dimensional road intersections, this is the largest order symmetry group.

Now let $\mathcal{M}$ be an embedding of a graph $G$ into a closed connected orientable surface $\S$ where $\S$ is itself
embedded into $\mathbb{R}^3$ (we will denote the latter fact by $\S \subset \mathbb{R}^3$).
For our second result we will work in the category of piecewise linear embeddings both for the graph in the surface, and the surface in the 3-dimensional space.
%We call $\mathcal{M}$ a 3-dimensional
%embedding of $G$ into $\S$ and denote it by $(G, \mathcal{M}, \S)$.
The \emph{image} $\mathcal{M}(G)$ of $G$ is defined as the subset of $\S$ which consists of the union of the set $\mathcal{M}(V(G))$ of all points  that $\mathcal{M}$ maps $V(G)$ to and 
the curves connecting these points which $\mathcal{M}$ maps the edges of $G$ to.
The \emph{isometry group} $\IG(\mathcal{M})$ of $\mathcal{M}$ %$(G, \mathcal{M}, \S)$ \m{introduce/find a name for triplet} %is the automorphism group of $(S, \mathcal{M}(G))$, which
consists of all Euclidean isometries $f: \mathbb{R}^3 \to \mathbb{R}^3$ which map $\S$ to $\S$,
%\m{isometry $implies$ orientation preserving?}
 $\mathcal{M}(G)$ to $\mathcal{M}(G)$ and $\mathcal{M}(V(G))$ to $\mathcal{M}(V(G))$. 
%{\color{orange} We say that $\mathcal{M}$ is \emph{piecewise smooth} if $\S$ can be triangulated so that
%each triangle is a differentiable manifold, the edges of each triangle are smooth, the vertices of $G$ coincide
%with the vertices of triangles, and the edges of $G$ are unions of edges
%of triangles.}
We say that an embedding $\mathcal{M}$ has
%{\color{orange} A subgroup of Euclidean isometries is called 
\emph{$n$-fold rotational symmetry} if \emph{both} the embedded graph and the underlying surface are invariant under the rotation $r$ by angle $2 \pi /n$ about some axis in $\mathbb{R}^3$, i.e., if $r \in \IG(\mathcal{M})$.
We call a road interchange $I=(G,H,\mathcal{M}, \S)$ \emph{3-dimensional} if $\S \subset \mathbb{R}^3$.
%and $\mathcal{M}$ is piecewise linear.
 \emph{The symmetry group $\IG(I)$ of $I$} consists of those $f \in \IG(\mathcal{M})$
that fix the face bounded by $H$. 
%%If each element of $\IG(I)$ preserves the colour of $G$, we call $\IG(I)$ \emph{colour preserving}.

%We will consider graph embeddings into surfaces $\S \subset \mathbb{R}^3$
%that are \emph{piecewise linear}, that is, $\S$ is a union of a finite number of flat triangles 
%and each embedded edge consists of a union of a finite number of line segments. 
%Our results almost trivially carry out to the piecewise differentiable setting.

%\hyphenation{dimen-sional}
The next is our main result about 3-dimensional road interchanges.
%Recall that $C_n$ denotes the cyclic symmetry group of order $n$.
\begin{theorem} \label{thm.geometric.opt}
    Let $n \ge 2$ be an integer. 
    Let $\mathcal{M}$ be a piecewise linear embedding of $K_{n,n}$ into a closed connected orientable surface $\S \subset \mathbb{R}^3$ such that there is a face $F_H$ bounded by a Hamiltonian cycle. 
    Suppose $\mathcal{M}$ has n-fold rotational symmetry that leaves the boundary of $F_H$ invariant.
    %If $\mathcal{M}$ is 2-cell
    Then the genus of $\S$ is at least $\Lgeom(n)$, where %\m{Is $n \lfloor \frac n 4 \rfloor-1_{n \equiv 3 (\bmod 4)}$ nicer?}
    \begin{align*}
    \Lgeom(n)=\begin{cases}
        \frac {n^2} 4 - 1, &\mbox{if }  n\equiv0\,(\bmod \, 4);  \\
        \frac {n(n-1)} 4, &\mbox{if }  n\equiv1\,(\bmod \, 4); \\
        \frac {n (n-2)} 4, &\mbox{if }  n\equiv2\,(\bmod \, 4);  \\
        \frac {n (n+1)} 4 - 1, &\mbox{if }  n\equiv3\,(\bmod \, 4).
    \end{cases}
    \end{align*}
    Furthermore, if $n \ne 4$, this lower bound is best possible.
\end{theorem}
%We note that using Euler's formula is not sufficient to prove the lower bound here.
Of course, $\Lcomb(n) \le \Lgeom(n)$, but note that equality holds only in certain cases.
A simple, but suboptimal lower bound for $\Lgeom(n)$ can be obtained from Theorem~\ref{thm.topsym} and the fact that not every orientable surface can be
embedded in $\mathbb{R}^3$ with a given symmetry group: in $\mathbb{R}^3$ we can only
have embeddings that have symmetry groups of a $t$-prism for certain $t > 0$, the Platonic solids and their subgroups  \cite{undine, tucker2014}.
The correct lower bound follows from a new combinatorial argument that exposes a local obstruction that prevents a corresponding `quotient embedding' into a smaller genus surface, see Section~\ref{sec.geom.lower2}. 

Our constructions that achieve genus $\Lgeom(n)$ can be called \emph{ring road interchanges}. For
each entering motorway we construct a block with some `simple bridges' (handles). We connect the blocks into a ring.
For $n \bmod 4 \in \{1,2\}$ the layout is particularly simple: in the inter-block connections the traffic always moves counter-clockwise and the lanes heading to each exit motorway are adjacent. For $n \bmod 4 \in \{0, 3\}$ and $n > 4$, the optimal constructions are similar, but they use a star-shaped bridge and some lanes where traffic between the blocks goes in the opposite, clockwise, direction, see Section~\ref{sec.geom.upper}.

%%Can a 3-dimensional $n$-way interchange $I= (G,H,\mathcal{M}, \S)$ have a symmetry group $A(I)$ of order larger than $n$?
%%The possible symmetry groups, under Euclidean isometry, for a \m{(closed, connected??)} surface $\S$ of genus $g$ embedded in $\mathbb{R}^3$ are the symmetry groups of the k-prism for certain $k$, the Platonic solids and their subgroups \cite{tucker2014}. Tucker \cite{tucker2014} determined for which $g$ a surface $\S$ can be embedded in $\mathbb{R}^3$ to have one of the above symmetry groups. 
%%%Only the n-prism symmetry groups can have arbitrarily large order. 
%%%However 
%%\m{isomorphic?} 
%%{\color{orange} A simple geometric argument shows that
%%since $A(I)$ fixes the face bounded by $H$, it must be
%%isomorphic to a subgroup of the dihedral group $Dih_{2n}$. %Furthermore, if $A(i)$ is isomorphic to $C_{2n}$ or the dihedral group $D_{2n}$,
%%If we furthermore ask that $A(I)$ is colour preserving (incoming lanes are always mapped to incoming lanes and outgoing lanes are mapped to outgoing lanes), then it must have $n$-fold rotational symmetry.}

The case $n=4$ is an exception in Theorem~\ref{thm.geometric.opt}. We show below that the minimum genus in this case is $4$, and one of the embeddings of such genus is the 4-way Pinavia interchange \cite{jbk2010}, shown in Figure~\ref{fig.pinavia_embedding} below.

The Hamiltonian cycle assumption can be relaxed in Theorem~\ref{thm.geometric.opt}. The result remains true if we only require that no vertex of $K_{n,n}$ is mapped to a fixed point and at least one point of the surface is a fixed point, see Lemma~\ref{lem.geometric.opt} below.

%Limitations: how practical is minimizing the genus? The radius problem. The star-shaped bridges.

\section{Definitions}

A directed multigraph $G$ is a pair $(V, E)$, where $V$ and $E$ are disjoint sets
together with a map $E \to V \times V$ assigning to each edge $e$ its endpoints $(u,v)$;
we say that $e$ goes from $u$ to $v$. We use $V(G)=V$ and $E(G) = E$.
When there is only one edge $e$ with endpoints $(u,v)$ we use notation $uv$ for $e$.
Below when necessary we consider any 2-coloured bipartite multigraph $G$ as directed,
by implicitly assuming all edges go from white vertices to black vertices.

For a directed multigraph $G$ we let $U(G)$ denote the underlying undirected multigraph.
We say that a directed multigraph $G_1$ and an undirected multigraph $G_2$ are \emph{isomorphic}
if $G_2$ is isomorphic to $U(G_1)$. %\m{where do I use?}

%We will use notation $\mathcal{M}: G \to \S$ for an embedding $\mathcal{M}$ 
Let $\mathcal{M}$ be an embedding
of a directed or undirected multigraph $G$ into a surface $\S$. We call a connected component of $\S \setminus \mathcal{M}(G)$ a \emph{region}.  If a region of $\mathcal{M}$ is 2-cell, i.e. if it is homeomorphic to a disk, we call it a \emph{face} of $\mathcal{M}$. A well known fact is that there is a bijection between 2-cell embeddings of $G$ (up to homeomorphism) and \emph{rotation systems} for $G$, which are sets $\{ \pi_v: v \in V(G)\}$ where $\pi_v$ is a cyclic ordering of edges incident to $v$ \cite{grosstucker1987}. 

Let $\Gamma=(\Gamma, \star)$ be a group.
Given a directed multigraph $G$ and a function $\alpha: E(G) \to \Gamma$, called a \emph{voltage function} or a \emph{voltage assignment},
the \emph{derived} graph $G'$ for voltage group $\Gamma$ is the directed multigraph $G'$ on the vertex set $V(G) \times \Gamma$ and edges 
$\{ ( (x,a), (y,a \star \alpha(e))) : e \in E(G),\, f(e) = (x,y)\}$. Here $f$ maps $e$ to its endpoints. $(G, \alpha)$ is
called the \emph{voltage graph} or the \emph{base graph} for $G'$, and we call $\Gamma$ the \emph{voltage group}. If $G'$ is derived from $(G, \alpha)$ and $\Gamma$, we will say that its undirected
version $U(G')$ also is.
%This is a specific case of the voltage graph construction, see \cite{grosstucker1987}.
If $\mathcal{M}$ is a 2-cell embedding of $G$ into some surface $\S$
then $(G, \mathcal{M}, \S, \alpha)$ is called an \emph{embedded voltage graph}.

The main point of the \emph{voltage graph} construction is that the embedding $\mathcal{M}$ of a base graph $G$ can be used to obtain a \emph{derived embedding} $\mathcal{M}'$ of the derived graph $G'$ \cite{grosstucker1987}.
Each face of $\mathcal{M}'$ is generated by precisely one face of $\mathcal{M}$.
Let $F$ be a face of $\mathcal{M}$ with boundary $e_1^{s_1} \dots e_k^{s_k}$.
Here $s_i = 1$ if the direction of $e_i$ agrees with the direction of the boundary walk of $F$ and $s_i=-1$ otherwise. 
The \emph{net voltage} of $F$ is $\alpha(e_1)^{s_1} \star \dots \star \alpha(e_k)^{s_k}$, where $x^{-1}$ denotes the inverse of $x$ in $\Gamma$.
We call the number of edges in the boundary walk of $F$ its \emph{size}. We call a face \emph{Hamiltonian} if its boundary is a Hamiltonian cycle. If $F$ has size $k$ and net voltage $g$, then there 
are $n/|g|$ faces in $\mathcal{M}'$ \emph{generated by} $F$, and each of these faces has size $k |g|$ \cite{grosstucker1987}. Here $|g|$ 
is the order of $g$ in $\Gamma$.
The main base graph we will meet in this paper is the dipole graph $D_n$, that is, the bipartite multigraph on two vertices $\{v_w, v_b\}$ and $n$ parallel directed edges from $v_w$ to $v_b$. Here $v_w$ and $v_b$ stand for `white' and `black' vertex respectively,
as we sometimes treat $D_n$ as 2-coloured. The only group $\Gamma$ we will meet below will be  $\mathbb{Z}_n$ (the cyclic group of order $n$).

Embedded voltage graphs obtained from $D_n$ can be alternatively
represented by transition graphs \cite{esz2006, ellinghamstephens2009, ellingham2014}.
For our purposes, a \emph{transition graph} of order $n$ is a directed (multi-)graph with vertex set $\mathbb{Z}_n$
and edge set consisting of the union of the edges
of two directed Hamiltonian cycles $C_1$ and $C_2$. 
%$C_1$ and $C_2$ encode the
%rotation at the white vertex and therotation at the black vertex of $D_n$ respectively and
%the vertex label on a cycle corresponds to the voltage of an edge.
 The
edges of $C_1$ are called \emph{solid} edges and the edges of $C_2$ are called
\emph{dotted} edges of $G$. An edge from $u$ to $v$ is denoted $u \solid v$ if it is solid,
and $u \dotted v$ if it is dotted.
The edges of $G$ are a disjoint union of edges of simple directed even cycles (i.e. without
repeated vertices) in $G$ with alternating edge type.
The \emph{net transition} $\alpha(C)$ of one such \emph{alternating cycle} (called \emph{boundary walk} in \cite{esz2006}) $C=(v_1, \dots, v_k)$ with edges $v_1 \solid v_2 \dotted v_3 \solid \dots v_k \dotted v_1$ 
is %an element of $\mathbb{Z}_n$ 
%given by
$\alpha(C) = - v_1 + v_2 - v_3  \dots  + v_k$.
%The transition graph is an alternative, more visual representation
%of a voltage graph on two vertices.

\section{Proofs for combinatorial embeddings}
\label{sec.combinatorial}

In this section we prove Theorem~\ref{thm.topsym}.
%
%
%A pair of bijections $(\sigma_1, \sigma_2)$, $\sigma_1: V(G) \to V(G)$ and $\sigma_2: E(G) \to E(G)$ is called an \emph{automorphism} for a 2-cell embedded directed multigraph $(G, \mathcal{M}, \S)$ if the rotation at $\sigma_1(v)$ is $(\sigma_2(e_{v,1}), \dots, \sigma_2(e_{v,k_v}))$ if and only if the rotation at $v$ is $(e_{v,1}, \dots, e_{v,k_v})$ for each $v \in V(G)$ and the endpoint map $f$ satisfies $f(e) = xy$ if and only if $\sigma_2(f(e)) = \sigma_1(x)\sigma_1(y)$. As is common, below we use the same symbol $\sigma$ for both $\sigma_1$ and $\sigma_2$.
%
Let $I=(G, H, \mathcal{M}, \S)$ be a complete $n$-way interchange.
%We can assume $H$ is directed by identifying it with the boundary walk of the outer face it bounds.
Write $H = (v_0, \dots,$  $v_{2n-1})$.
 Suppose
\begin{itemize}
 \item[($\star$)] $\mathcal{M}$ has an orientation preserving automorphism $\sigma$ that maps $v_i$ to $v_{i+2}$ for each $i \in \mathbb{Z}_n$.\label{prop.star}
\end{itemize}
 Then $\mathcal{M}$ is derived from an embedded voltage graph with two vertices and $m$ parallel edges, $m \ge n$, voltage group $\Gamma = {Z}_n$  and a voltage assignment $\alpha$ such that for each $x \in \mathbb{Z}_n$ there is $e \in E(B)$ with $\alpha(e) = x$, i.e. $\alpha$ is a surjection. It follows from ($\star$) that the base embedding can be chosen so that the face bounded by $H$ is generated by a face $F_0$ of size two.
We can remove edges with repetitive voltage that do not lie on $F_0$ without increasing the genus until there are exactly $n$ edges
of different voltage left. So for the minimum genus problem, we may assume $G$ is $K_{n,n}$ and $m=n$. Conversely, each graph derived from $D_n$ with $\Gamma=\mathbb{Z}_n$ and a bijective voltage assignment $\alpha: E(D_n) \to \mathbb{Z}_n$ such that at least one face of the derived graph is bounded by a Hamiltonian cycle and generated by a face of size 2 has property ($\star$).

\subsection{Even \texorpdfstring{$n$}{Lg}}

\begin{proofof}{Theorem~\ref{thm.topsym}, even $n$. }
Let $(D_n, \Ga, \S_0, \alpha)$ be the base embedded voltage graph that yields $\mathcal{M}$.
Let $F_0, F_1, \dots, F_k$ denote the faces of the $\Ga$ and for each face $F_i$ let $k_i$ and $g_i$ denote the size and net voltage of $F_i$.
%(we will consider $\Z_n$ elements as remainders modulo $n$)
Consider the sums $\sum _{i=0}^k k_i$ and $\sum _{i=0}^k g_i$ : in both sums each edge is taken into account twice (once in both directions), therefore
\begin{align}
\sum_{i=0}^k k_i=2n  \quad\mbox{and}\quad \sum_{i=0}^k g_i \equiv 0 \pmod{n}. \label{voltages}
\end{align}
For every $i=0,1,\dots, k$ the face $F_i$ generates $f_i=\frac{n}{|g_i|}$ faces in the derived graph, each of which has 
size
%a facial boundary of a length 
$k_i \cdot |g_i|$.
%We note that
%\begin{align*}
%f_i = \begin{cases}
%\gcd(n, g_i) &\text{ if } g_i \neq 0 \\
%n &\text{ if } g_i = 0
%\end{cases}
%\end{align*}
%Without loss of generality we can assume that $F_0$ generates Hamiltonian faces in the derived graph embedding.
In order to obtain the lower bound for the genus of this embedding, we will seek to maximize the sum $f=\sum_{i=0}^k f_i$.
Fix a face $F_0$ of $\Ga$ that generates Hamiltonian faces in $\mathcal{M}$.
We call a face $F$, $F \ne F_0$ of $\Ga$ \emph{optimal} if it generates quadrangular faces in $\mathcal{M}$. All other faces $F$, $F \ne F_0$
will be called non-optimal. Without loss of generality we assume that
$F_1, \dots, F_l$ are non-optimal and $F_{l+1}, \dots, F_k$ are optimal.

$F_i$ is optimal iff $k_i \cdot |g_i|=4$. So we can split the optimal faces $F_i$ into two cases: %In our case $k_i$ is always an even number, so we split such faces $F_i, \, i \geq 1$ into two cases:
\begin{enumerate}[(1)]
\item $g_i \equiv 0$ ($f_i=n$) and $k_i=4$;
\item $g_i \equiv \frac{n}{2}$ ($f_i=\frac{n}{2}$) and $k_i=2$.
\end{enumerate}
%All other faces $F_i, \, i \geq 1$ will be called \emph{non-optimal}. 
%We denote non-optimal faces by $F'_1, F'_2, \dots, F'_l$ and lengths and voltages of their facial boundaries by $n'_1, n'_2, \dots, n'_l$ and $g'_1, g'_2, \dots, g'_l$ respectively. 
We denote $\sum_{j=1}^l k_j=2m$. For non-optimal faces $F_j$ we have $k_j \cdot |g_j| \geq 6$, hence we can split them into three cases:
\begin{enumerate}[1)]
\item $|g_j|=1, \, k_j \geq 6$ and $f_j=n \leq k_j \cdot \frac{n}{6}$;
\item $|g_j|=2, \, k_j \geq 4$ and $f_j=\frac{n}{2} \leq k_j \cdot \frac{n}{8}$;
%\item $\gcd(n, g'_j)< \frac{n}{2}$ and $f'_j \leq \frac{n}{3} \leq n'_j \cdot \frac{n}{6}$.
\item $|g_j| > 2$ and $f_j \leq \frac{n}{3} \leq k_j \cdot \frac{n}{6}$.
\end{enumerate}
Therefore
\begin{align}
\sum_{j=1}^l f_j \leq \sum_{j=1}^l k_j \cdot \frac{n}{6}=\frac{mn}{3}. \label{onethird}
\end{align}

Now we rewrite \eqref{voltages} as follows:
\begin{align}
0 \equiv \sum_{i=0}^k g_i &\equiv g_0 + \sum_{j=1}^l g_j + \frac{2n-k_0-2m}{2} \cdot \frac{n}{2} \nonumber \\ 
& \equiv \begin{cases}
g_0 + \sum_{j=1}^l g_j &\text{ if } k_0/2+m \text{ is even},  \\
g_0 + \sum_{j=1}^l g_j+\frac{n}{2} &\text{ if } k_0/2+m \text{ is odd}.
\end{cases} \label{hamvolt}
\end{align}
We can also rewrite the sum $f$:
\begin{align}
f=\sum_{i=0}^k f_i = f_0 + \sum_{j=1}^l f_j + \frac{2n-k_0-2m}{2} \cdot \frac{n}{2}. \label{faces}
\end{align}

First assume that $k_0=2$. Then $|g_0|=n$, $f_0=1$ and from \eqref{hamvolt} we have that $g_0$ cannot equal to 0 or $\frac{n}{2}$ when $n>2$, hence there is at least one non-optimal face in $\Ga$ (i.e. $l>0$). Let $\hat{\Ga}$ be an embedded voltage graph which gives the lowest derived graph embedding genus among all graphs $\Ga$ satisfying theorem conditions with $k_0=2$. 

Now for even $n$ it is easy to get a transition graph that yields an embedding of $K_{n,n}$ with two Hamiltonian
faces and optimal genus $n (n-2)/4$. This can be done
by generalizing examples of Ellingham \cite{ellingham2014}. %\m{ar gerai?}
In particular, we can take the cycles $(C_1, C_2)$ where $C_1$ is
\[
\left(\frac n 2, \frac n 2 -1, \dots, 1, \frac n 2 + 1, \frac n 2 + 2, \dots, n-1, 0\right)
\]
and $C_2$ is
\[ 
\begin{cases}
    \left(1, 2, -2, 4, -4, \dots, - \left(\frac n 2 -1\right), 0, -1, 3, -3, 5, -5, \dots, \frac n 2\right) &\mbox{if $n \equiv 2 \,(\bmod\,4)$}, \\ 
    \left(1, 2, -2, 4, -4, \dots, \frac n 2, 0, -1, 3, -3, 5, -5, \dots,  - \left(\frac n 2 -1\right) \right) &\mbox{if $n \equiv 0 \,(\bmod\,4)$}.  
\end{cases}
\]
Note that in each case the two Hamiltonian faces are generated by two alternating 2-cycles, or 2-faces in the corresponding voltage graph: they are $(2,1)$ and $(-1, 0)$.

%Then from construction with two Hamiltonian faces ref{??} we have a lower boundary
From this construction we have a lower bound  $\hat{f}$  for the number of faces in the embedding derived from $\hat{\Ga}$:
\begin{align}
\hat{f}= 1+\sum_{j=1}^l f_j+\frac{n^2-n-mn}{2} \geq 1+1+\frac{n^2-n-n}{2} \label{upper}
\end{align}

It suffices to prove that in \eqref{upper} the equality holds. From \eqref{upper} and \eqref{onethird}
\begin{align*}
1+\frac{(m-1)n}{2} &\leq \sum_{j=1}^l f_j \leq \frac{mn}{3}, \mbox{ therefore}\\  6 & \leq n(3-m).
\end{align*}

As $m>0$, $m$ equals to 1 or 2. If $m=1$, then $l=1$ and from \eqref{hamvolt} we have that $g_1 \equiv -g_0 \pmod n$ so $f_1=1$ and the equality in \eqref{upper} holds.

Similarly, if $m=2$ and $l=1$, then  $1+\frac{n}{2} \leq \sum_{j=1}^l f_j=1$, a contradiction.
%Yet, if we let $l=1$, then $1+\frac{n}{2} \leq f_1 \leq \frac{n}{2}$, contradiction.

We are left with the case when $m=2$ and $l=2$. We have that $k_1=k_2=2$ and $f_j \leq \frac{n}{3}$ for $j=1, \, 2$.
As $1+\frac{n}{2} \leq f_1+f_2$, one of $f_1$ and $f_2$ must equal to $\frac{n}{3}$ and the other one belongs to $\{\frac{n}{3}, \frac{n}{4}, \frac{n}{5}\}$. Without loss of generality let $f_1=\frac{n}{3}$.

If $\gcd(\frac{n}{2}, f_1, f_2)=r>1$, then from \eqref{faces} we get $\gcd (n, f_0) \geq r >1$, a contradiction. Therefore $\gcd(\frac{n}{2}, f_1, f_2)=1$ and when $f_2$ equals to $\frac{n}{3}$,  $\frac{n}{4}$ or $\frac{n}{5}$, then $n$ must be 6, 12 or 30 respectively. In all these cases the equality \eqref{upper} holds.

This ends our proof when $k_0=2$. We have seen that in this case the maximum number of faces $\hat{f}$ does not exceed the one of construction having all quadrangular faces except from two Hamiltonian faces which has a genus $g = \frac{n(n-2)}{4}$. As any construction with $k_0>2$ yields an embedding with at least two Hamiltonian faces, it also has a genus not smaller than $\frac{n(n-2)}{4}$.

\end{proofof}

\subsection{Odd \texorpdfstring{$n$}{Lg}}

The proof for odd $n$ is a bit trickier; $\Lcomb(n)$ itself is not as simple as for even $n$. We start with two lemmas that give examples of optimal genus.
%??\m{parity in $\mathbb{Z}_n$..}

\begin{lemma}\label{lem.tg}
    For each odd $n$, $n \ge 3$, there is a transition graph  that
    yields an embedding of $K_{n,n}$  of genus $\lfloor \frac{n(n-1)}4  \rfloor$ with a Hamiltonian face.
\end{lemma}

\begin{proof}
    A construction for %all $n$ with
     $n\equiv 3\,(\bmod\, 4)$ %of genus $\lfloor n(n-1)/4 \rfloor$ 
    is as follows.
    For $n=3$, we can trivially take a transition graph with cycles $C_1=(0,1,2)$ and
    $C_2=(0,2,1)$ which yields an embedding of $K_{3,3}$ with three Hamiltonian faces.
    We generalize this to $n\ge 7$ by taking
    \begin{align*}
      C_1=&\left(0,1, \dots, \frac{n-1}  2, n-1, n-2, \dots, \frac{n+1} 2\right) \mbox{ and} \\
      C_2=&\left(0, \frac{n+1} 2, 2, \frac{n+1} 2+2, 4, \dots, n-1,\right.  \\
           &\quad\quad\left.\frac{n-1} 2, n-2, \frac{n-1} 2-2, n-4, \dots, 1 \right).
    \end{align*}
    This transition graph has three alternating cycles of
    length two: $(0,1)$, $((n+1)/2, 0)$ and $((n-1)/2, n-1)$.
    Their net transition is $1$, $(n-1)/2$ and $(n+1)/2$, respectively,
    and all these elements have order $n$ in $\mathbb{Z}_n$.
    The remaining alternating cycles are of the form 
%    $(x, x+1, (n+1)/2-x-2, (n+1)/2-x-1)$ or
%    $(x+1, x+2, (n+1)/2-x-1, (n+1)/2-x)$,
    %$(x, x+1, n-x, n-x-1)$ 
%    $(x, x+1, n-x-2, n-x-1)$ or
%    $(x+1, x+2, n-x-1, n-x)$,
    $(x, x+1, (n+1)/2+x+2, (n+1)/2+x+1)$ or
    $(x+1, x+2, (n+1)/2+x, (n+1)/2+x+1)$,
    thus they have size 4 and net transition 0. 
    The derived graph has 3 Hamiltonian faces and all other faces of size $4$.
    % and excess $6n-8$.

    A transition graph for $n \equiv 1\,(\bmod\,4)$
    is
    \begin{align*}
    C_1 = &\left(\frac {n-1} 2, \dots, 1, \frac{n+1} 2, \dots, n-1, 0\right) \mbox { and} \\
    C_2 = &\left(\frac{n+1} 2, 0, n-1, \frac{n-1} 2 - 2, n-3, \frac{n-1} 2 - 4, \dots, \frac {n+1} 2 + 1,\right. \\
           &\quad\quad\left.1, \frac{n+1} 2 + 2, 3, \frac{n+1} 2 + 4, \dots, \frac{n-1} 2 - 1, \frac{n-1} 2\right).  
    \end{align*}

    \begin{figure}[h]
    \centering
    \begin{subfigure}[b]{0.45\textwidth}
    \centering
    \includegraphics[width=0.4\linewidth]{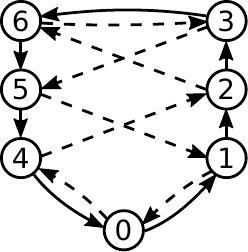}
    \end{subfigure}
    \begin{subfigure}[b]{0.45\textwidth}
    \centering
    \includegraphics[width=0.4\linewidth]{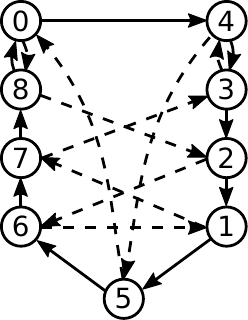}
    \end{subfigure}
    \caption{\label{fig.tg_odd} Transition graphs for $n=7$ and $n=9$ as in the proof of Lemma~\ref{lem.tg}.}
    \end{figure}

    Again it is easy to see that this transition graph has two alternating cycles $(n-1, 0)$ and $((n-1)/2, (n-1)/2-1)$,
    an alternating cycle $((n+1)/2, (n+1)/2 + 1, 1, (n+1)/2, 0, (n-1)/2)$ of length 6 and net transition $-n \equiv 0\,(\bmod\, n)$ and all other alternating cycles of length 4 and net transition 0. Thus
    the derived embedding has two Hamiltonian faces, $n$ faces of size $6$
    and all other faces of size 4.
    %and excess $4n - 4$ which corresponds
    By Euler's formula its genus is $\Lcomb(n) = \lfloor n (n-1)/4 \rfloor$. Both cases are illustrated in Figure~\ref{fig.tg_odd}.
\end{proof}

\begin{lemma}\label{lem.opt_g1g2}
   Let $n \ge 25$, $n \equiv 3\,(\bmod\,4)$. 
   Suppose  $g_1, g_2 \in \mathbb{Z}_n \setminus \{0\}$ and $g_1 + g_2 + 1 = 0$.
   Then there exists a transition graph on $n$ vertices
   with 3 alternating 2-cycles of net transition $1$, $g_1$ and $g_2$
   respectively, and all other alternating cycles of size 4
   and net transition 0. 
\end{lemma}

The corresponding derived graph is $K_{n,n}$ with one Hamiltonian face,
$n/|g_i|$ faces of size $2|g_i|$ for $i\in\{1,2\}$ and all other faces of size 4.
This graph has genus $\lfloor \frac{n(n-1)} 4 \rfloor + 1 - \frac 1 2 ( \frac n {|g_1|} + \frac n {|g_2|})$.

\bigskip

\begin{proof}
%We describe not just one, but $e^{\Omega(n \ln n)}$ number of constructions. 
Consider
solutions of
\begin{align}\label{eq.abcde}
\begin{cases}
    a = g_1 - b\\
    d = g_2 - c\\
    e = b + c \
\end{cases}
\end{align}
with unknowns $a,b,c,d,e \in \mathbb{Z}_n$, operations modulo $n$ and a restriction that $a,b,c,d,e, -a,-b,-c,-d,-e$ are all distinct and avoid $\{0, 1, -1\}$.

%We claim that this system of equations always has a solution. 
%Indeed, let $b$ be arbitrary from $\mathbb{Z}_n \setminus \{0, -1, 1, -g_1, -g_1 -1, -g_1 + 1\}$.
%Pick $c$ uniformly at random from $\mathbb{Z}_n \setminus \{0,-1,1, -g_2,$ $-g_2-1, -g_2+1,$ $-a, a,$ $-b, b,$ $-g_2 -a, -g_2+a,$ $-g_2 - b, -g_2 + b\}$.
%Define $a$, $d$ and $e$ according to (\ref{eq.abcde}). Then using the definition and the fact that $n$ is odd so that $g + x \ne -x$ for any $g,x \in \mathbb{Z}_n$, we have that $a,b,c,d,-a,-b,-c,-d$ are all distinct and neither is in $\{0, -1, 1\}$.
%Since $b$ and $c$ are non-zero $e \ne b$ and $e \ne c$, and since $b \ne -c$, $e \ne 0$. As $n$ is odd, $\mathbb{P}(e = x) \le 1/(n-14)$ for $x \in \{-a, a, -d, d, -b, -c, -1,1\}$.
%

We claim that (\ref{eq.abcde}) always has a feasible solution. 
Indeed, let $b$ and $c$ be taken independently and uniformly at random from $\mathbb{Z}_n$.
Define $a$, $d$ and $e$ according to (\ref{eq.abcde}).
Let $A$ be the event that $-a,a$,$-b,b$, $0,-1,1$ are all distinct, and additionally $b \ne -g_2$.
Since $n$ is odd, $x \in \mathbb{Z}_n \setminus \{0\}$ implies $x \ne -x$.
Also for $y \in \mathbb{Z}_n$ there is a unique solution of $2x = y$, which we denote $y/2$.

$A$ occurs if and only if $b \not \in \{0, \pm 1, -g_2\}$ and $a \not \in \{0, \pm 1, \pm b\}$,
which is equivalent to the event $b \not \in \{0, -1, 1,$ $g_1, g_1 -1,$ $g_1 + 1, g_1/2, -g_2\}$. So
 \[
 \pr(A) \ge 1 - \frac 8 n.
\]
Similarly, let $B$ be the event that $a$, $b$, $c$, $d$, $e$, $1$ and their negations are distinct and non-zero.
For $a$ and $b$ such that $A$ occurs
\begin{align*}
%\pr(B | a,b) = \pr(c \not \in 0,-1, 1, g_2, g_2-1, g_2+1,-a, a, -b, b, -b-1, -b+1, -b-a, -b+a, g_2 -a, g_2+a, g_2 - b, g_2 + b)
%\pr(B | a,b) &= \pr(c \not \in \{0, \pm 1, g_2, g_2 \pm 1, \pm a, \pm b, -b \pm 1,  -b \pm a, -b \pm b, -b \pm c, -b \pm d})
%             &=\pr(
\pr(B | a,b) &= \pr(c \not \in \{0, \pm 1, \pm a, \pm b\} \cap d \not \in \{0, \pm 1, \pm a, \pm b, \pm c\} \cap \\
             &\quad\quad\quad    e \not \in
                    \{0, \pm 1, \pm a, \pm b, \pm c, \pm d\}) \\
&=  \pr(c \not \in \{0, \pm 1, \pm a, \pm b,  g_2, g_2 \pm 1, g_2 \pm a, g_2 \pm b, g_2 - g_1 +b, g_2/2, \\
&    \quad\quad\quad    -b, -b \pm 1,  -b \pm a, -2b,  -b/2, (g_2 - b)/2\}) \ge 1 - \frac {24} n.
\end{align*}
%%We had some simplification since $a=-d \Leftrightarrow e=-1$, $d \ne -c$ since $g_2 \ne 0$, $e \ne c$ since $b \ne 0$,
%%$e \ne -d$ since $b \ne -g_2$.
So for $n \ge 25$
\[
\pr(B) = \pr(B|A) \pr(A) \ge \left(1- \frac 8 n\right) \left(1 - \frac {24} n\right) > 0.
\]
%
%Pick $c$ uniformly at random from $\mathbb{Z}_n \setminus \{0,-1,1, -g_2,$ $-g_2-1, -g_2+1,$ $-a, a,$ $-b, b,$ $-g_2 -a, -g_2+a,$ $-g_2 - b, -g_2 + b\}$.
% Then using the definition and the fact that $n$ is odd so that $g + x \ne -x$ for any $g,x \in \mathbb{Z}_n$, we have that $a,b,c,d,-a,-b,-c,-d$ are all distinct and neither is in $\{0, -1, 1\}$.
%Since $b$ and $c$ are non-zero $e \ne b$ and $e \ne c$, and since $b \ne -c$, $e \ne 0$. As $n$ is odd, $\mathbb{P}(e = x) \le 1/(n-14)$ for $x \in \{-a, a, -d, d, -b, -c, -1,1\}$.
%
%
%Thus probability that $a,b,c,d,e,$ $-a,-b,-c,-d,-e$ and $0,-1,1$ are all distinct is 
%at least
%\[
%1 - \frac 8 {n-14} > 0.
%\]
%Indeed, let $b,c$ drawn independently and uniformly at
%random from $\{2, \dots, n-2\}$, and let $a,d,e$ be defined according to (\ref{eq.abcde}). 
%There are $\binom 5 2 = 10$ pairs  $\{x, y\}$ of distinct variables from $\{a,b,c,d,e\}$.
%Since they are nonzero, $a \ne b$, $c \ne d$, $e \ne b$ and $e \ne c$.
%For each of the remaining six pairs $\{x,y\}$ we have $\mathbb{P}(x=y) \le (n-3)^{-1}$.  Thus the probability that no two variables coincide is at least $1 - 6/(n-3) > 0$. 
%%Note that we could alternatively choose first $b$ then $c$ deterministically by avoiding the above defined sets.
We further consider $(a,b,c,d,e)$ as a fixed solution of (\ref{eq.abcde}).

Let $t=(n-1)/2$.
Pick (we can always do it; in fact, we can do it in $e^{\Omega(n \ln n)}$ ways) a sequence $s_1=(v_0, v_1, \dots, v_t)$, such that
\begin{enumerate}
\item For each $x \in \mathbb{Z}_n \setminus \{0\}$, $s_1$ contains exactly one of $\{-x, x\}$.
\item $v_0=0$, $v_1=1$, $v_k = a$, $v_{k+1}=b$, $v_l= c$, $v_{l+1} = d$ and $v_t = e$, where $k$ and $l$ are even and $1 < k < l < t$.
\end{enumerate}

Define $s_2=(-v_1, \dots, -v_t)$. Then the elements of $s_1$ and $s_2$ partition $\mathbb{Z}_n$.
We will construct the transition graph by adding solid and dotted edges to an empty graph on vertex set $\mathbb{Z}_n$.
First add directed paths of solid directed edges 
\begin{align*}
&P_1=(v_0,v_1,\dots, v_k),\\ 
&P_2=(-v_1, -v_2, \dots, -v_k, v_{k+1}, v_{k+2}, \dots, v_l), \\
&P_3=(-v_{k+1}, -v_{k+2}, \dots, -v_l, v_{l+1}, v_{l+2}, \dots, v_t) \quad \mbox{and} \\
&P_4=(-v_{l+1}, -v_{l+2}, \dots, -v_t).
\end{align*}
Next, add %solid directed edges $(-v_k, v_{k+1}) = (-a,b)$, $(-v_l, v_{l+1}) =(-c,d)$ and
dotted edges $(v_{k+1}, -v_k) $, $(v_{l+1}, -v_l)$ and $(v_1,v_0)$
(i.e., the edges $(b, -a)$, $(d, -c)$ and $(1,0)$).
These edges complete three alternating 2-cycles of net transition $a+b = g_1$, $c+d = g_2$ and $1+0=1$ respectively,
see Figure~\ref{fig.g1g2_odd}.

Now for every $i \in \{1, \dots, t-1\} \setminus \{k, l\}$ add dotted directed edges $(v_{i+1}, -v_i)$ and $(-v_{i+1}, v_i)$,
this yields an alternating 4-cycle  $v_i \solid v_{i+1} \dotted -v_i \solid -v_{i+1} \dotted v_i$ of net transition $v_{i+1} + v_i - v_{i+1} - v_i = 0$. 

Since $k$ and $l$ are even and $t$ is odd, all the dotted edges added so far yield the following paths that go backwards alternating between $s_1$ and $s_2$: 
\begin{align*}
&P_1'=(-v_t, v_{t-1}, -v_{t-2}, \dots, -v_{l+1}), \\
&P_2' =(v_t, -v_{t-1}, v_{t-1}, \dots, \\
       & \quad \quad \quad \quad\quad\quad v_{l+1}, -v_l, v_{l-1}, -v_{l-2}, \dots, v_{k+1}, -v_k, v_{k-1}, \dots, -v_2, v_1, v_0), \\
&P_3'=(v_l, -v_{l-1}, \dots, -v_{k+1}) \quad \mbox{and} \\
&P_4'=(v_k, -v_{k-1}, \dots, -v_1).
\end{align*}

\begin{figure}
    \centering
    \includegraphics[width=0.75\linewidth]{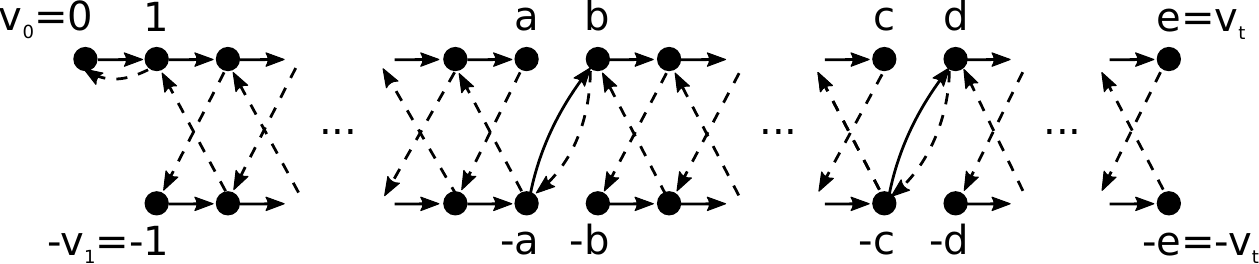}
    \caption{\label{fig.g1g2_odd} Construction of the proof of Lemma~\ref{lem.opt_g1g2} (adding cycles $A_1$ and $A_2$ would complete the transition graph).}
\end{figure}

Finally, add two alternating 4-cycles: 
\begin{align*}
&A_1= c \solid {-b} \dotted {-e} \solid 0 \dotted c = (v_l, -v_{k+1}, -v_t, v_0)\quad\mbox{and} \\
&A_2 = e \solid {-d} \dotted a \solid {-1} \dotted e = (v_t, -v_{l+1}, v_k, -v_1).
\end{align*}
Notice that the edges of $A_1$ and $A_2$ complete a solid Hamiltonian cycle $(P_1P_2P_3P_4)$ and a dotted Hamiltonian cycle $(P_1'P_4'P_2'P_3')$, where for paths $A$, $B$ we use notation $AB$ to denote the concatenation of their vertices in the obvious way.

To complete the proof, we have to check that $A_1$ and $A_2$ both have net transition $0$. The net transition of $A_1$ is $-c - b + e = 0$ since $e=b+c$. The net transition of $A_2$ is 
\[-e - d - a -1 = -(a+d + e +1) = - (a+d+b+c+1) = -(-1 + 1)=0.\] In the last step we used $a+b+c+d = g_1+g_2 = -1$, which follows from the first two equations of (\ref{eq.abcde}) and the theorem's assumption.
\end{proof}

\bigskip

We are now ready to complete the proof of Theorem~\ref{thm.topsym}.

\medskip

\begin{proofof}{Theorem~\ref{thm.topsym}, odd $n$.}
    We continue with notation from the proof for even $n$.
    We first prove the lower bound. In the end of this proof we will show that it is optimal.

    Let $s_1, \dots, s_f$ be the sizes of the faces of $\mathcal{M}$. 
    Let $v$, $e$, and $g$ denote the number of vertices,
    the number of edges and the genus of $\mathcal{M}$ respectively. 
    Denote the \emph{excess} of $\mathcal{M}$ by
    \[
    \ex = \sum_i (s_i - 4) = 2e - 4f =  2 n^2 - 4 f.
    \]
    By Euler's formula $\ex = 8g - 2e + 4v - 8 = 8g - 2n^2 + 8n - 8$.
     So using the definition of $\Lcomb(n)$, in order to prove the inequality stated in the theorem, it suffices
    to show 
%    \begin{align}\label{eq.ex_to_show}
%    \ex \ge  \begin{cases}
%                         6n-8, &\mbox{if } n \equiv 1\,(\bmod\,4) \\
%                         6n - 4 - 4(p_1 + \frac n {p_1}), &\mbox{if } n \equiv 3\,(\bmod\,4) \mbox{ and } n \bmod p_1^2 \ne 0; \\
%                         6n- 12, &\mbox {otherwise}.
%                \end{cases}
%    \end{align}
    \begin{align}\label{eq.ex_to_show}
    \ex \ge  \begin{cases}
                         6n - 4 - 4\left(p_1 + \frac n {p_1}\right), &\mbox{if } n \equiv 3\,(\bmod\,4), n \ne p_1 \mbox{ and } p_1^2 \nmid n, \\
                         6n - 4 - 4\left(1 + \frac n {p_1}\right), &\mbox{if } n \equiv 3\,(\bmod\,4) \mbox{ and } (p_1^2 \mid n \mbox { or } n=p_1), \\
                         6n - 16, &\mbox{if } n \equiv 1\,(\bmod\,4), 3 \mid n \mbox { and } 9 \nmid n, \\
                         6n - 8, &\mbox{if } n \equiv 1\,(\bmod\,4) \mbox{ and } (3 \nmid n \mbox { or } 9 \mid n). \\
                \end{cases}
    \end{align}
    Since %%$K_{n,n}$ is simple and 
    $g$ is integer, we must have
    \begin{equation} \label{eq.div8}
         \begin{cases}
                         6n-8 - \ex \equiv 0\,(\bmod\,8) &\mbox{for } n \equiv 1\,(\bmod\,4), \\
                         6n - 12 - \ex \equiv 0\,(\bmod\,8) &\mbox{for } n \equiv 3\,(\bmod\,4).
          \end{cases}
    \end{equation} 
    With voltage group $\mathbb{Z}_n$ a face $F$ of size $k$ and voltage $h$
    % in the base graph corresponds to $n/|h|$ faces of length $k |h|$ in the derived graph. Thus $F$ 
    contributes
    \begin{equation}\label{eq.ex_face}
       \ex(F) = \frac n {|h|} (k|h| - 4) = n k - \frac {4n} {|h|}  
    \end{equation}
    to the excess of the derived embedding. Note that since each face $F$ in $\mathcal{M}$ has size at least 4, $\ex(F) \ge 0$.

%Let us now prove (\ref{eq.ex_to_show}). %In order for $\mathcal{M}$ to be an embedding of $K_{n,n}$,
%for each $k \in \mathbb{Z}_n$ there must be an edge from $w$ to $b$ with voltage $k$ modulo $n$.
Suppose %embedding $\mathcal{\tilde M}$ of the dipole $D_n$ with a voltage function $\phi: E(D_n) \to \mathbb{Z}_n$ yields an embedding $\mathcal{M}$ of $K_{n,n}$ as in
there exists an embedding $\mathcal{M}$ of $K_{n,n}$ as in the statement of the theorem,
but with genus smaller than $\Lcomb(n)$, i.e. suppose that (\ref{eq.ex_to_show}) does not hold. Let $B=(D_n, \Ga, \S_0, \alpha)$ be the corresponding base graph as before.
Note that the faces of $\Ga$ have even sizes and the orders of their voltages are divisors of $n$, in particular they are odd.
%As before, let $F_0$ be a face of the base graph which generates at least one Hamiltonian face in $\mathcal{M}$, $k_0$ is its size and $g_0$ is its voltage.
Then $|g_0| = 2n/{k_0}$ and by (\ref{eq.ex_face}) $\ex \ge \ex(F_0) = k_0(n-2)$. For $k_0 \ge 8$ (\ref{eq.ex_to_show}) holds.
%
% As the genus $g$ of $\mathcal{M}$ is integer,
%it still holds with $k_0 = 6$, where we can have equality in (\ref{eq.ex_to_show}) only in the case
%$n\equiv3\,(\bmod\,4)$ and $n \equiv 0\,(\bmod\,3)$. 
It cannot be $k_0=6$ either, since in this case $3 \mid n$, so $n$ is not prime and
by (\ref{eq.div8}), (\ref{eq.ex_to_show}) still holds. The only case where we could have equality in (\ref{eq.ex_to_show}) for $k_0=6$ is $n \equiv 1\,(\bmod\,4)$ and $9 \mid n$.

% In this case $3 | n$, so by (7) and (8), the equality  is only possible in the case n=1 (mod 4) and n = 0 (mod 9).

Also $k_0 \ne 4$, since $4 \nmid 2n$. Thus for $B$ to be a counterexample 
we must have $k_0 = 2$, $|g_0|=n$ and $\ex(F_0) = 2n-4$.
Since orders of elements are preserved under a group automorphism,
we can assume without loss of generality that $g_0=1$.

%, and if it has a face of size $4$
%and voltage $g$ with $|g| \ge 7$, $\ex \ge 2n-4 + 4n - 4n/7$

For any face of $\Ga$, we have that $\ex(F) = 0$ if and only if $F$ has size $4$ and voltage $0$. As before, we call such faces $F$ \emph{optimal}, and the remaining faces $F$, $F \ne F_0$ non-optimal. We assume that  $F_1, \dots, F_l$ are all the non-optimal faces of $B$, and using $g_0 = 1$ and 
%For $i\in\{0,\dots,l\}$ denote by $g_i$ the voltage (sum of arc voltages along the boundary walk) of $F_i$.
%As before, \m{sulyginti žymėjimus}
(\ref{voltages}) we have
\begin{equation}\label{eq.gsum}
    1 + g_1 + \dots + g_l = 0.
\end{equation}
For a non-optimal face we have $g_i = jn/|g_i|$ for some $j \in \{0, |g_i| - 1\}$, so
%,
%we must have that 
\begin{equation}\label{eq.gcd}
\gcd(n/|g_1|, \dots, n/|g_l|)=1.
\end{equation}
%Since $g_0 = 1$, at least one $g_i$ with $i \ge 0$ must be non-zero. 
%Note that for non-optimal $F_i$ of size at most $4$, then $g_i \ne 0$
Let $F_i$ be non-optimal.
As $K_{n,n}$ has no repeated edges,
$g_i$ can only be zero if $k_i \ge 6$.
Using (\ref{eq.ex_face}) we have $\ex(F_i) \ge 2n-4n/3 = 2n/3$.
Thus if $B$ is a counterexample to (\ref{eq.ex_to_show}), then $2n-4 + 2 l n/3 \le \ex< 6n - 8$ which implies  $l \le 5$.
If $F_i$ has size $k_i \ge 8$, then 
\[
\ex \ge \ex(F_0) + \ex(F_i) \ge 2n -4 + k_i n - 4n \ge 6n -4
\]
and again (\ref{eq.ex_to_show}) holds. Similarly,
there cannot be any face of size $6$ and voltage $g$ with $g \bmod\,n \ne 0$, since in this case
$|g| \ge 3$ and $\ex \ge 2n-4 + 6n - 4n/3  > 6n-8$.

%Using the above observations and (\ref{eq.ex_face}) we can conclude that if (\ref{eq.ex_to_show}) fails, then: 
Using (\ref{eq.gsum}), the fact that the right side of (\ref{eq.ex_face}) increases with $|h|$ and the above observations
we conclude that (\ref{eq.ex_to_show}) holds with inequality unless perhaps
\begin{enumerate}
  %\item if some non-optimal $F_i$ has size $6$ then $g_i=0$, $l=3$ and the other non-optimal faces $F_j$, $j \ge 1$ have size 2 and $|g_j| \in \{3,5\}$;
  \item some non-optimal face, say $F_1$, has size 6, $g_1=0$, $l = 3$, the other non-optimal faces $F_j$, $j \ge 2$ have size 2 and $|g_j| \in \{3, 5\}$ (in the case $l=2$, $|g_2| = n$ by (\ref{eq.gsum}), so (\ref{eq.ex_to_show}) holds with equality) or
  \item some non-optimal face, say $F_1$, has size $4$, $l = 2$ and $g_1,g_2 \in \{3,5\}$ (if $l=1$ and $|g_1|=n$ and (\ref{eq.ex_to_show}) holds with equality).
%  \item if all $F_i$ have size $2$ then either $t \le 2$ or $|g_j| \in \{3,\dots, 11\}$.
\end{enumerate} 
Note that (\ref{eq.gcd}) implies that $n \le \prod_{i=1}^l |g_i|$. In particular, if a counterexample $B$  has 
a non-optimal face which has size 4 or 6, then $n \le 15$ and it can be easily checked using (\ref{eq.ex_face}) only 
that (\ref{eq.ex_to_show}) holds for all such $n$. %\m{patikrinti!}

So we may further assume that all non-optimal $F_i$ have $k_i = 2$. By (\ref{voltages}) $2n = 2 (l+1) + 4 n_{opt}$
where $n_{opt}$ is the number of optimal faces in $\Ga$. %and $|F_i|$ is the size (number of edges in the boundary) of $F_i$, 
So $l = n - 1 - 2n_{opt}$ is even. By (\ref{eq.gsum}) $l>0$, therefore  $l \in \{2,4\}$.

Let us assume that $n=|g_0| \ge |g_1| \ge \dots \ge |g_l|$. 

If $l=4$, then $(|g_1|, \dots, |g_l|)$ cannot equal $(5,5,3,3)$ 
or $(7,5,3,3)$ as in these cases by (\ref{eq.gcd}) $n$ is 15 or 105 respectively, so by (\ref{eq.ex_face}) the inequality (\ref{eq.ex_to_show}) holds.
%As (\ref{eq.ex_face}) is monotone in $|g|$, 
For $(|g_1|, .., |g_l|) = (q, 5, 3, 3)$ with $q \ge 9$, we have by (\ref{eq.ex_face}), that $\ex > 6n-4$. Since the right side of (\ref{eq.ex_face}) is increasing in $|h|$, it follows 
by (\ref{eq.gcd}) that
 %using (\ref{eq.ex_face}) and (\ref{eq.gcd}) we have 
 $|g_1|=q$ and $|g_2|=|g_3|=|g_4| = 3$,
where either $q=n$ or $n= 3 q$ for some $q$ not a multiple of $3$.
In the former case $\ex = 2 (2n-4) + 3 (2n - 4n/3) = 6n -8$, so (\ref{eq.ex_to_show}) holds with equality.
In the latter case
\[
\ex = 2n - 4  + 2n - 4n/q + 3(2n - 4n/3) = 6n - 16.
\]
Using (\ref{eq.div8}) this is only possible for $n \equiv 1\,(\bmod\,4)$, $3 \mid n$ and $9 \nmid n$, so in
both cases (\ref{eq.ex_to_show}) holds with equality as well.
%In the case  $q=n$, we have no contradiction to (\ref{eq.ex_to_show}),
%while in the case $q=n/3$ the genus of $\mathcal{M}$ is $\lceil n(n-1)/4 \rceil - 2$.

It remains to consider the case $l=2$. Suppose
$|g_1| = q$ and $|g_2|=p$ with $p \le q$. Recall that $p \ge 3$. By (\ref{eq.gcd})
we have $p q = n \gcd(p,q)$ and
\[
    \ex = 2n-4 + 2n - \frac {4n} q + 2n - \frac {4n} p = 6n - 4 - 4 a(p,q).
\]
where $a(p,q) = n/q + n/p$. As  $a(p,q) = (p+q) / \gcd(p,q)$, it is even, so by (\ref{eq.div8}) such excess is only possible for $n \equiv 3\,(\bmod\,4)$.

%We claim that $\frac {(p+q)} {\gcd (p,q)}$ is maximized (and genus of $\mathcal{M}$ minimized)
%when $p=p_1$ and $q=n/p_1$ in case $n \mod p_1^2 \ne 0$ and when $p=p_1$ and $q=n$ otherwise. 
Thus
\[
\ex \ge 6n - 4 - 4 \max_{(p,q)\in S} a(p,q)
\]
where the set $S$ consists pairs  $(p,q)$  of 
%prime 
divisors of $n$ with $\gcd(n/p, n/q) = 1$ and $p, q > 1$.

For a prime $y$ and a positive integer $x$, let $\mul(x,y)$ denote the largest
integer $k$ such that $y^k \mid x$. (\ref{eq.gcd}) implies that for each prime 
factor $r$ of $n$ either $\mul(p,r)=\mul(n,r)$ or $\mul(q,r) = \mul(n,r)$.

First note that $a(p, n) \le a(p_1, n)$ for any divisor $p$ of $n$ such that $p > 1$. $(p,q)=(p_1,n)$ will be our first candidate to minimize $\ex$. For such $(p,q)$ we have $\ex = 6n - 4 - 4 (1 + n/p_1)$ as in the second case of (\ref{eq.ex_to_show}). %Then $(p,q) \in S$ for any divisor $p$ of $n$ which is greater than 1.
%In this case $a(p,q)$ is maximized $p=p_1$, the smallest prime divisor of $n$. 

Now suppose $q < n$.
If $\gcd(p,q) > 1$, let $r$ be a prime that divides both $q$ and $p$.
If $\mul(q,r) = \mul(n,r)$ define $p' = p/r^{\mul(p,r)}$ and notice
that $a(p',q) > a(p,q)$ and $(p',q) \in S$. (To see that $p'$ cannot equal $1$, note that this would
imply $\mul(q,r) = \mul(n,r)$ for all $r \mid n$, so $q=n$.) Similarly, if $\mul(p,r) = \mul(n,r)$
we can define $q' = q/\mul(q,r)$ and get $a(p,q') > a(p,q)$.
Repeating this for all $r$ that divide $\gcd(p,q)$ we obtain
$(p',q') \in S$ such that $a(p',q') > a(p,q)$, $\gcd(p', q')=1$ and so $p' q'=n$.

Without loss of generality we can further focus on the set $S_1$ of pairs $(p,q) \in S$ with $p \le q$, $pq = n$ and $\gcd(p,q) = 1$. When $p$ has at least two unique prime divisors, say $p=r^k b$ where $r$ is prime, $b>1$ and $r \nmid b$, 
we can define $p' = p/r^k$ and $q'=q r^k$. Then $a(p',q') = p/r^k + q r^k > p + q = a(p,q)$.
Thus if $q < n$
\[
   a(p,q) \le \max_{r} a(r^{\mul(n,r)}, n/r^{\mul(n,r)}) 
\]
where the maximum is over all prime divisors of $n$.
 
Suppose the smallest prime divisor $p_1$ of $n$ has multiplicity 1. %For $q=n$ we have $a(p,q)=a(p,n) \le a(p_1, n)$.
If $n$ is prime, then $p_1=p=q=n$ and $a(p,q) = 2$ is the only possibility, in which case (\ref{eq.ex_to_show}) holds with equality.
Otherwise, for $(p,q) \in S_1$, we have $p_1 \le p  \le \sqrt n$.
%It cannot be $n=p_1$ as $q < n$. % in this case $p=q=n$, and we settled the case $q=n$ above. %$a(p,q)=2$ and (\ref{eq.ex_to_show}) follows.
Since $f(s) = a(s, n/s) = s + n/s$ is decreasing in the interval $(0, \sqrt n\,]$,
we have $a(p_1, n/p_1) > a(p, n/p)$ when $p > p_1$. 
So when $\mul(p_1, n) =1$ and $(p,q) \in S_1$ %, $1 < p \le q \le n$
%ar čia aš dar darau prielaidą q<n ar nebedarau??? Lyg nebedarau???
%ne, vis dar darau. Pirminiai atskirai, nes čia eilė = 1 negalima
\[
  a(p,q) \le \max(a(p_1, n/p_1), a(p_1, n)) = a(p_1, n/p_1).
\]
Now suppose $p_1^2$ divides $n$. Then a maximum $a(p,q)$ over $(p,q) \in S_1$ could be attained either
with $p= r^t$ or $q=r^t$ for some prime divisor $r$ of $n$, $t=\mul(n,r)$ and $q = n/p$ or with $p=p_1$ and $q=n$.
We claim that it is always the latter. % the latter choice always gives maximum $a(p,q)$.
To see this, first note that this is true for $n$ a prime power.
If $n$ is not a prime power, take $r$ that maximizes $a(r^t, n/r^t)$. Note
that  $r^t < b$, where $b=n/r^t$, otherwise we could
replace $r$ with another prime divisor $r_2$ of $n$ and get a larger value of $a$.
So $r < b$ and for $t \ge 2$
\begin{align*}
  a\left(r, n\right) - a\left(r^t, \frac n {r^t} \right) &= b r^{t-1} + 1 - b - r^t  \\
  &=(r^{t-1} - 1) (b-r) - (r-1) > 0
\end{align*}
and $a(r^t,n/r^t) < a(r,n) \le a(p_1,n)$. For $t=1$, writing $n=p_1^2 r x$ we get
\begin{align*}
 a(p_1, n) - a(r, n/r)  &= p_1 r x + 1 - p_1^2 x - r \\
 &=p_1(r-p_1)x + 1-r \ge p_1(r-p_1) + (1-r) \\
 &=(p_1 - 1) (r-p_1-1) > 0.
\end{align*}
%since $r^l < b$ implies $r < b$.
In both cases $a(r^t, n/r^t) < a(p_1, n)$. This finishes the proof of (\ref{eq.ex_to_show}).

Let us show that (\ref{eq.ex_to_show}) is best possible.
If $n$ is prime or $n \equiv 1 \,(\bmod\,4)$ and ($3 \nmid n$ or $9 \mid n$),
Lemma~\ref{lem.tg} gives a construction of genus $\Lcomb(n)$.

Suppose $n \equiv 3 \,(\bmod\,4)$
 and $p$ is not prime. The first few such $n$ 
are $15, 27, 35$.
 If $p_1^2 \nmid n$, let $(p,q) = (p_1, n/p_1)$,
otherwise let $(p,q) = (p_1, n)$.
Since $n/p$ and $n/q$ are relatively prime, by Euclid's algorithm we can find $g_1, g_2 \in \mathbb{Z}_n$
 such that 
 \begin{equation}\label{eq.gg1}
 g_1 + g_2 + 1 \equiv 0 \, (\bmod\,n), \quad |g_1|=p \quad \mbox{and} \quad |g_2| = q.
 \end{equation}
For $n=15$ the transition graph corresponding to genus $\Lcomb(15)=49$ embedding
is given by %a solid and dotted Hamiltonian cycles $(_1, H_2)$,
\begin{align*}
& C_1 = (0, 1, 2, -1, -2, 3, 4, -3, -4, 5, 6, 7, -5, -6, -7) \\
& C_2 = (7, -6, 5, -4, 3, -2, 1, 0, 4, -3, -7, 6, -5, 2, -1).
\end{align*}
It was found by computer and inspired Lemma~\ref{lem.opt_g1g2}.
It corresponds to choices $g_1 = 5$, $g_2=9$ $a= (g_1-1)/2, b=a+1, c=(g_2 -1)/2, d=c+1$ and $s_1 = (0,1, \dots, 7)$ in that
lemma. For other $n$, since they are larger than $25$, we can apply Lemma~\ref{lem.opt_g1g2}.

Finally, suppose $n \equiv 1 \,(\bmod \, 4)$, $3 \mid n$ but $9 \nmid n$. The smallest
such $n$ are $21, 33, 57$. Let  
$(p,q) = (3, n/3)$ and define $g_1, g_2$ by solving (\ref{eq.gg1}).
Define $g_3=n/3$ and $g_4 = -n/3$.  Then
\[
   g_1 + g_2 + g_3 + g_4 + 1 = 0.
\]
Construct a transition graph 
as in the proof of Lemma~\ref{lem.opt_g1g2}, but with the following modification:
after choosing suitable $a$, $b$, $c$, $d$ and $e$, choose $f \in \mathbb{Z}_n$, such
that $f \not \in \{ {-x}, x\}$ and $f \not \in \{x-n/3, {-x}-n/3\}$ for
 $x \in \{0, 1, a, b, c, d, e\}$, and $-f \ne f+n/3$. We have at least $n-27$ possibilities for $f$. Assuming $n > 27$, we construct a transition graph as in the proof of Lemma~\ref{lem.opt_g1g2} 
but additionally require that $v_m = f$ and $v_{m+1}=f+n/3$ for
some $m$, $l+1 < m < t$, where $l, m$ and $t$ are as in Lemma~\ref{lem.opt_g1g2}. This construction has the same properties
as the construction for $n=4k+3$, except now $t$ is even
and the paths $P_3'$ and $P_4'$ have their endpoints switched,
so the dotted edges make up two cycles instead of a single Hamiltonian cycle. We can remedy this
by replacing the dotted edges $(v_{m+1}, -v_{m})$
and $(-v_{m+1}, v_{m})$ with
$(v_{m+1}, - v_{m})$ and $(-v_{m+1}, -v_{m})$. This
eliminates one alternating 4-cycle and introduces two new 2-cycles
$(v_{m}, v_{m+1})$ and $(-v_{m}, -v_{m+1})$ with net transition
$n/3$ and $-n/3$ respectively. Thus in the corresponding voltage graph we have 2-faces of orders $n/3, 3, 3,3,3$
and genus $\Lcomb(n)$ as required. Finally, a similar construction
also exists for $n=21$. The transition graph is given by 
\begin{align*}
%& C_1=(0, 1, 2, -7, -8, 3, 4, -9, -10, 5, 6, 7, 10, 9, 8) \\
%& C_2=(7, 9, 5, -10, 3, -8, 1, 0, 4, -9, 8, 6, 10, 2, -7).
C_1 = (0, 1, 5, -1, -5, 9, 7, 6, 8, -9, -7, -6, -8, -2, 3, 10, -4, 2, -3, -10, 4) \\
C_2 = (4, 10, 3, 2, 5, -1, -4, -10, -3, -2, -8, 6, -7, 9, -5, 1, 0, 8, -6, 7, -9)
\end{align*}
\end{proofof}

\section{Proofs for 3-dimensional embeddings}
\subsection{Proof of the lower bound: the geometric part}
\label{sec.geom.lower1}
%{\color{orange}
All surfaces in this section will be embedded in $\mathbb{R}^3$, compact and
piecewise linear.
A \emph{piecewise linear curve}, or simply a \emph{curve}, on a surface $\S$ is
a piecewise linear map $f: [0,1] \to \S$.
 We call a curve a \emph{simple arc}, or simply an \emph{arc}, if
it is injective. We call it a \emph{simple closed curve}, or simply a \emph{closed curve},
if it is injective on $[0,1]$ with an exception that  $f(0) = f(1)$.
    We call a curve \emph{simple} if it is an arc or a closed curve.
     The endpoints of a (closed or non-closed) curve $f$ are $f(0)$ and $f(1)$,
    and we say that $f$ goes from $f(0)$ to $f(1)$. We say that a curve $f$ is \emph{internally disjoint}
    from a  set $X$ if $f$ is closed or $f(x) \not \in X$ for all $x \in (0,1)$.
%The set of endpoints is defined to be empty for a closed curve,
%    however, when we obtain a closed curve from an embedded loop of a multigraph, we assume it has one endpoint $f(0)$.
%defined analogously, except we have $f(0) = f(1)$. Since
%all arcs and closed curves we deal from now on a simple and piecewise smooth, we refer to them
%simply as \emph{arcs} and \emph{closed curves} respectively.
For convenience, we make no distinction between a curve $f$ and its image $f([0,1])$ when it is clear from the context.

Let $a$ and $b$ be two simple curves on a surface $\S$ such that no endpoint belongs to the other curve and $a \cap b$ is finite.
Then each point in $x \in a \cap b$ can be classified as a touching or crossing, see \cite{juvanmalnicmohar1996, malnicnedela1995} and references therein. 
%\m{More basic literature?} 
Fix an orientation of $\S$. Define the \textit{intersection sign} $\sigma(x,a,b) \in \{-1, 0, 1\}$ as shown in Figure~\ref{fig.intersection_type}.
\begin{figure}[h]
\centering
\includegraphics[scale=0.75]{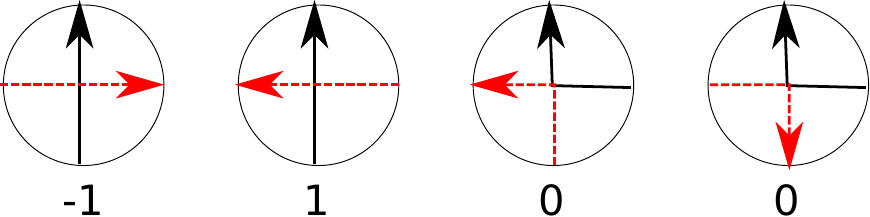}
\caption{The sign $\sigma(x,a,b)$ of the intersection point $x$ between a solid black curve $a$ and a dotted red curve $b$.}
\label{fig.intersection_type}
\end{figure}
More formally, in a small neighbourhood $U \subset \S$ of $x$, we can represent $a$ as a segment (i.e., an arc) $a_{in}$ that ends at $x$ and a segment $a_{out}$ that
starts at $x$, so that $a_{in} \cap a_{out} = \{x\}$ and $a_{in} \cup a_{out} = a \cap U$. Similarly we can define $b_{in}$ and $b_{out}$. We set 
$\sigma(x, a,b) = -1$, if the cyclical clockwise ordering around $x$ of the segments incident to $x$ is $(a_{in}, b_{in}, a_{out}, b_{out})$. We set $\sigma(x, a,b)=1$ if the ordering is $(b_{out}, a_{out}, b_{in}, a_{in})$. Otherwise, we set $\sigma(x,a,b)=0$ (i.e. the intersection is a touching). In the case $x$ is an endpoint of either curve we may define $\sigma(x,a,b) = 0$; this case
can always be avoided by a small local modification.

 %assigned the \emph{crossing type} $\sigma(x,a,b) \in \{-1, 0, 1\}$ as follows \dots
%\m{Apibrėžti}.

%Let $\S \subset \mathbb{R}^3$ be a closed connected orientable surface, l
Let $a$ be a simple curve on an oriented surface $\S$
and $C$ a finite set of pairwise disjoint simple curves on $\S$. %, internally disjoint from the endpoints of $a$.
% in $\S$, such that endpoints of each arc in $\{a\} \cup C$ are disjoint from
%all other curves.
 Suppose $a$ shares with each curve in $C$ a finite number of points. 
We define 
\[
\sigma(a, C) = \sum_{b \in C} \sum_{x \in a \cap b} \sigma(x, a,b).
\] 
%This definition can be extended
%also for the case where $a$ does not necessarily share a finite number of points with curves in $C$.

Now let $\mathcal{M}$ be an embedding of a directed multigraph $G$ into a connected oriented surface $\S \subset \mathbb{R}^3$,
and let $C$ be a set of pairwise disjoint simple curves on $\S$. %, internally disjoint from $\mathcal{M}(V(G))$. 
%Consider a set of curves $\mathcal{E} = \{\mathcal{M}(e) : e\in E(G)\}$.
For $e \in E(G)$ let $\mathcal{M}(e)$ be the simple curve on $\S$ assigned to $e$ by $\mathcal{M}$. 
Given a positive integer $n$, the \emph{voltage assignment} function $\alpha_{\mathcal{M}, C, n}: E(G) \to \mathbb{Z}_n$ is \emph{well defined}
 if for each edge $e \in E(G)$, $\mathcal{M}(e)$ shares with each curve in $C$ a finite number of points and $\mathcal{M}(G)$ is disjoint from the endpoints of each curve in $C$. 
 In that case for $e \in E(G)$ we set
\begin{equation}\label{eq.cutvolt}
\alpha_{\mathcal{M},C, n}(e) = \sigma(\mathcal{M}(e), C) \bmod n.
\end{equation}
%}

%%open%% {\color{orange} In knot theory, a set of closed curves in $\mathbb{R}^3$ is called \emph{unlink} if \dots}

To study symmetric embeddings, and more generally immersions, of a surface into $\mathbb{R}^3$
we can use the following approach, see \cite{undine}. Given a surface $\S' \subset \mathbb{R}^3$ which is invariant under the action of a finite subgroup $\Gamma$ of
orientation preserving Euclidean isometries in $\mathbb{R}^3$, we can define the quotient mapping $q$ from $\mathbb{R}^3$ to the orbifold $\mathbb{R}^3 / \Gamma$, the quotient space obtained by identifying each orbit of $\Gamma$ to a single point and equipped with a suitable topology. For $n$-fold rotational symmetry the corresponding quotient \emph{surface}, or orbifold, $\S' / \Gamma$ 
is a closed %, connected
 and orientable surface, and the map $q$ restricted to $S$ is a branched
covering of $\S' / \Gamma$ with an even number of branch points. 
 
We additionally have an embedding of a graph $G'$ into $\S'$ and we ask that this embedding is invariant under $\Gamma$. As above, we can then define an embedding of a quotient graph $G$ in $\S' / \Gamma$. %, the base embedded graph for $G'$.
We will give a proof of the next rather straightforward proposition as we were unable to find a suitable equivalent in the literature.
%We additionally have an embedding $\mathcal{M}': G \to \mathbb{S}'$ with $\IG(\mathcal{M}')$ containing $\Gamma$, and
%we will study the quotient graph $G'$ $q(\mathcal{M}'(G'))$, which corresponds to the embedded voltage graph in $\mathbb{S} / \Gamma$. We
%will represent $\mathbb{S} / \Gamma$ as a piecewise linear embedding in $\mathbb{R}^3$. Since $q$ is very simple for rotational symmetry,
%we prove the next proposition only using elementary geometry.
%Since the main idea of the lower bounds uses a simple geometric property, we formulate and prove the next technical fact we need about orbifolds in elementary geometric terms.

\begin{prop}\label{prop.cuts}
Let $n\ge 3$ and $g' \ge 0$ be integers, and let $G$ be a directed connected multigraph and $G'$ a connected multigraph. 
%The statements below are equivalent.
The two statements below are equivalent.

%but do I need the embedding as well?
\begin{enumerate}[(a)]
\item There is a piecewise linear embedding $\mathcal{M}'$ of $G'$ into a closed connected
oriented surface $\S' \subset \mathbb{R}^3$ of genus $g'$ such that $\mathcal{M}'$ has $n$-fold rotational symmetry,
some vertex of $\S'$ is a fixed point under the rotation, but no vertex of $G'$ is mapped to
a fixed point. 
\item There is a piecewise linear embedding $\mathcal{M}$ of $G$ into a closed connected oriented surface $\S \subset \mathbb{R}^3$
together with a set $C$ of pairwise disjoint arcs or closed curves on $\S$, 
%internally disjoint from $\mathcal{M}(V(G))$,  %this is now inside well defined
such that 
%%open%% the closed curves in $C$ {\color{orange} ???form an unlink???},
 the number of arcs (non-closed curves) in $C$ is $t$, $t \ge 1$, the genus $g$ of $\S$ satisfies $n g + (n-1) (t-1)=g'$,  $\alpha_{\mathcal{M}, C,n}$ is well defined and $G'$ is isomorphic to the graph derived from the voltage graph $(G, \alpha_{\mathcal{M}, C,n})$ and voltage group $\mathbb{Z}_n$. 
\end{enumerate}

\end{prop}

For a simple example with $G'$ a complete bipartite graph $K_{3,3}$ and $G$ a dipole $D_3$ see Figure~\ref{fig.sym3}. %We note that it is possible to drop the requirement in Proposition~\ref{prop.cuts} that no vertex is mapped to a fixed point. To do this, an extension of the voltage graph construction would be necessary.
In the special case when $\mathcal{M}$ is 2-cell, $\mathcal{M}'$ is just the derived embedding from the embedded voltage graph  $(G, \mathcal{M}, \S, \alpha_{\mathcal{M}, C,n})$ realised in $\mathbb{R}^3$ and voltage group $\Gamma=\mathbb{Z}_n$.
When $\S'$ has no fixed points under the rotation, our proof below shows that (a) still implies (b) with $t=0$, however in the opposite direction we can end up with a disconnected surface $\S'$ even if $\S$ is connected.

%\m{Can we drop the requirement that vertices are not fixed points by extending the voltage graph construction?}
\bigskip

\begin{proof}
%Let $c$ be a boundary cycle of an oriented surface $\S \subset \mathbb{R}^3$ with boundary. We will treat 
In this proof we only deal with oriented surfaces $\S \subset \mathbb{R}^3$ where each boundary component can be represented by a
closed curve $c$ on $\S$. An arbitrary point $x \in c$ can be chosen as both the start and the end point. $c$ can be given two possible directions. The intersection of $c$ with a sufficiently small neighbourhood around $x$ (which is homeomorphic to a half-disk) in $\S$ can be represented by a segment $c_{out}$ outgoing from $x$ and a segment $c_{in}$ incoming to $x$. The clockwise ordering of these segments at $x$ does not depend on $x$. If the ordering is $(c_{out}, c_{in})$ we say that \emph{the points of $\S$ are on the right of $c$}. Otherwise, if the ordering is $(c_{in}, c_{out})$, we say that \emph{the points of $\S$ are on the left} of $c$. 

% The direction where the points of $\S$ are on the \emph{right} of $c$ is the one  the counter-clockwise walk
%  along the boundary of a sufficiently small half disk around each point $x \in c$.

\textit {Proof of $(a)\!\implies\!(b)$.}
Let $a$ be the $n$-fold rotation symmetry axis of $\mathcal{M}'$.
%Let $I$ and $a$ be as in Proposition~\ref{prop.rotation}. 
Let $r: \mathbb{R}^3 \to \mathbb{R}^3$ be the rotation around $a$ by $2 \pi / n$. 
Fix any $n$ half-planes $P_0, \dots, P_{n-1}$
containing $a$ as their common boundary, such that $r(P_j) = P_{ (j+1) \bmod n}$. Denote $X_j=P_j \cap \S'$.
$P_0, \dots, P_{n-1}$ make up the boundaries of $n$ infinite prisms $\Pi_0, \dots, \Pi_{n-1}$ with pairwise disjoint interiors, where 
$\Pi_j$ is bounded by (and includes the points of) $P_j$ and $P_{ (j+1) \bmod n}$.
%Let $\Pi_j$ be the infinite prism 
%between $P_j$ and $P_{ (j+1) \bmod n}$, 
%i.e. the set of all points obtained by rotating $P_j$ by an angle $\phi$  around $a$ for each $\phi \in [0, 2 \pi /n]$.
Let $\S_j = \Pi_j \cap \S'$. 

Since $\mathcal{M}'$ is piecewise linear and $\S'$ is a surface,
%\begin{prop}
we can choose $P_0, \dots, P_{n-1}$ so that
    \begin{itemize}
       % \item $S_j$ contains one black and one white vertex.
       \item For $j \in \{0, \dots, n-1\}$, $X_j$ is a union of a finite number of (images of) pairwise disjoint simple curves in $\S_j$ (and $\S'$). We assume the direction
       of each curve $c \in X_j$ is chosen so that the points of $\S_j$ are on the right of $c$. 
       %{\color{orange} Each curve can be assigned direction, so that the points ``inside'' $\S'$ are on its ``left''.}
       %\m{Apibrėžti kreivių kryptį: paviršiaus ,,vidų`` atitinkantys taškai visada ,,kairėje`` pusėje.} 
       \item Each arc in $X_0$ starts and ends in $a$, and each curve in $X_0$ is internally
       disjoint from $a$. %closed curve in $X_0$ is disjoint from $a$.
       Define $B$ by $B=\S\cap a = X_0 \cap a$. Then $|B|$ is even and $|B|=2t$ where $t$ is the number of arcs  in $X_0$.

        \item The curves in $X_0$ are internally disjoint from $\mathcal{M}'(V(G'))$.
        %\item The endpoints of each curve in $X_0$ are disjoint from $\mathcal{M}'(G')$.
 
        \item For each $e \in E(G')$, the curve $\mathcal{M'}(e)$ shares with $X_0$ a finite number of points.
        
        %No image under $\mathcal{M}$ of a vertex of $G$ lies on $X_0 \setminus a$, the image
        %of each edge of $G$ intersects $X_0$ a finite number of times.
  \end{itemize}

 %With this choice of $P_0, \dots, P_{n-1}$ the following holds.

Now for each $e \in E(G')$, the curve $c_e=\mathcal{M}'(e)$  is disjoint from $B$.
Indeed, suppose on the contrary that $a$ and $\mathcal{M}'(e)$ contain
a point $x$ in common. Then $x$ must belong to the interior of $c_e$, since $\mathcal{M}'(V(G'))$ contains no fixed points. If $c_e$ is tangent to $a$ at $x$,
then, by considering the orbit (under the rotation) of the region incident to $e$, we see that since $n\ge 3$, $\S$ is not locally homeomorphic to the Euclidean plane at $x$ (see, e.g., proof of Theorem 78.1 of \cite{munkres}). Otherwise, by considering the orbit of $c_e$ in the neighbourhood of $x$, we see that $\mathcal{M}'$ is not a proper embedding, a contradiction.

% \begin{itemize}
         %\item 
         Identifying each point $x \in X_0$ with $r(x) \in X_1$ we obtain from $\S_0$ a quotient surface $\S$.        %This surface can be easily embedded in 
        $\S$ can be represented as a piecewise linear embedding in $\mathbb{R}^3$, for example,
        as the image of
        $\S_0$
        under a piecewise linear approximation $\tilde{f}$ of a function $f: \Pi_0 \to \mathbb{R}^3$, where $f$ rotates $x$ around $a$ by angle $n \theta$ if $x$ is at angle $\theta$
        with $P_0$. % Note that $f$ is a bijection except it maps both $P_0$ and $P_1$ to $P_0$.
         We further assume that $\S = \tilde{f}(\S_0)$.
         %, and the embedding $\mathcal{M}$ is defined accordingly.
         Note that $C = \tilde{f}(X_0) := \{\tilde{f} \circ c : c \in X_0\}$ is then a finite set of arcs (we call them \emph{cuts}) and 
        closed curves (we call them \emph{cut loops}) on $\S$. 
        %\item 
        
        Let $p$ map $x \in \mathbb{R}^3$ to the unique point $p(x) \in \Pi_0 \setminus P_1$ in the orbit of $x$ under the $n$-fold rotation and let $h = \tilde{f} \circ p$. Restricted to $\S'$, $h$ is a branched covering, the covered surface is $\S$ and $B$ is the set of branch points, each of order $n$. It is easy to see that $\S$ is  closed and orientable. Since $\S'$ contains a fixed point under the rotation, it follows that $\S$ is connected. 

Since $\mathcal{M}'(V(G'))$ avoids fixed points,  we have that this set is mapped to a set $V$ of size $|V(G')| / n$ by $h$.

For any $e \in E(G')$ $\mathcal{M}'(e)$ can be considered a simple curve, so we can define
a function $m_e$ by $m_e(x) = h(\mathcal{M}'(e)(x))$  $x \in [0,1]$. Since $\mathcal{M}'(e)$ avoids $B$,
$m_e$ is a simple curve.

%consider $m_e = h(\mathcal{M}'(e))$. Since $\mathcal{M}'(e)$ avoids $B$, $h$ is a local homeomorphism for each $x \in m_e$. It follows that $m_e$ is an arc with endpoints $h(s), h(t) \in V$ where $s$ and $t$ are the endpoints of $\mathcal{M}'(e)$. Choose the direction of $m_e$ arbitrarily, so
%that $m_e$ goes, say, from $h(s)$ to $h(t)$, and let the direction of $\mathcal{M}'(e)$ be defined by setting $m_e  = h \circ \mathcal{M}'(e)$.

As $n \ge 3$, the orbit of $\mathcal{M}'(e)$ under the rotation corresponds to $n$ distinct edges of $G'$, all of which are mapped to $m_e$ by $h$ (switch the direction of each such $\mathcal{M}'(f)$  to match the direction of $m_e$). Thus the images $m_e$ for $e \in E(G)$ define an embedding $\mathcal{M}$ of a connected directed graph $G$ with $|V(G')|/n$ vertices mapped to $V$ and with $|E(G)|/n$ edges into $\S$.

Now for any $e \in E(G')$ with $e=st$ and $m_e$ going from $h(s)$ to $h(t)$ follow the curve $\mathcal{M}'(e)$ from $s$ to $t$. By the definition of $\sigma$, for any $i \in \mathbb{Z}_n$, $s \in \{-1,1\}$ we cross from $\S_i$ to $\S_{ (i+s) \bmod\,n}$ at a point $x \in \S'$ if and only if $\sigma(h(x), m_e, h \circ c) = s$, where $c$ is
the curve of $X_i$ containing $x$. Summing over all points of intersection of $m_e$ with $X_0$ in $\S$, we get that an edge $e \in E(G')$ with $\mathcal{M}'(e)$ an arc from $r^{\beta}(u)$ to $r^{\gamma}(v)$ corresponds to an arc $m_e$ from $u$ to $v$ in $G$ such that $\gamma = \beta + \sigma(m_e, C) \bmod n$. 
So $G'$ is the graph derived from $(G, \alpha_{\mathcal{M}, C, n})$ and the voltage group $\mathbb{Z}_n$. 

%{\color{orange}
        %\item 
        %Let $h: \S' \to \S$ be the covering projection. 
%}        
        %\item 
Since $h$ has $2t$ branch points of order $n$, by the Riemann-Hurwitz formula, see, e.g., \cite{tucker2014},
%written with mistake in undine
the genus $g$ of $\S$ satisfies $g' = ng + (n-1) (t-1)$.

%    \end{itemize}

\medskip

  \textit{Proof of $(b)\!\implies\!(a)$.}
  %Let $\tilde{\S}$ be the surface obtained by repeatedly cutting
  Cut $\S$ at each $c \in C$.
  %, $\tilde{\S} = \S \setminus \left(\cup_{c \in C} c\right)$. 
  We obtain a possibly disconnected surface $\tilde{S}$ with boundary, where each closed curve $c$ corresponds to two boundary components
  and each arc in $C$ to one boundary component.
  More precisely, by the surface classification theorem (e.g., \cite{munkres}), there exists a piecewise linear map $h: \mathbb{R}^3 \to \mathbb{R}^3$ and a piecewise linear surface $\tilde{\S}$ with boundary
  such that $h(\tilde{\S}) = \S$, restricted to 
  $\tilde{\S} \setminus \cup_{c \in C} c$, $h$ is a homeomorphism and $\tilde{\S}$ is a sphere with some number $a \le g$ of simple handles and $b=2(|C|-t) + t = 2|C| - t$ boundary components.

  Fix an orientation of $\tilde{\S}$ which agrees with the orientation of $\S$ in the sense that rotations of any embedded graph are preserved
  under $h$. The direction of a curve $c \in C$ induces the direction of the boundary cycles of $\tilde{\S}$: for a closed curve we 
  obtain one boundary cycle $c_L$, $h(c_L) = c$ where the points of $\tilde{\S}$ are on the right of $c_L$.
  Similarly we have another cycle $c_R'$, $h(c_R) = c$, where points of $\tilde{\S}$ are on the left of $c_R$. % (the direction of counter-clockwise walks along the boundary of a small disk centered on $x \in c_R$ agrees with the direction of $c_R$).
  For an arc $c$ we have only one boundary cycle $c'$ in $\tilde{\S}$ such that $h(c') = c$. We assume that $c'$ is 
  formed from two arcs $c_L'$ and $c_R'$ which have a common start point (mapped to the start point of $c$ by $h$) and a common endpoint (mapped to the endpoint of $c$ by $h$), i.e. $c' = c_{R}' * (c_{L}')^{-1}$, where $*$ denotes the concatenation of two curves and ${}^{-1}$ denotes inversion of the direction.

  We can carry out the following simple modification of $\tilde{\S}$. First take a half-plane $P_0$ in $\mathbb{R}^3$ with boundary, say, the $z$ axis. Let $r$ be the rotation around $z$ by $2\pi/n$. On $P_0$ place $|C|-t$ disks disjoint from $z$ and $t$ half-disks centered at the $z$ axis, all pairwise disjoint. For each of the $|C|-t$ closed curves $c \in C$ connect the $i$th boundary cycle $c_L$ using a cylindrical tube, disjoint from the rest of the surface, to the boundary $B_i$ of the $i$th disk on $P_0$ and similarly connect $c_R$ to $r(B_i)$ on the half plane $P_1=r(P_0)$. For $j$th arc $c \in C$, $j \in \{1, \dots, t\}$, use a cylindrical tube to connect $c'$ with the closed curve formed from the boundary $B_j'$ of the $j$th half disk on $P_0$ and its rotated image $r(B_j')$. It is easy to ensure the added parts are piecewise linear and the modified
  surface is contained in the infinite prism bounded by $P_0$ and $P_1$.
   So we can further assume that $\tilde{\S}$ has all boundary cycles on $P_0$ and $P_1$ as described above and
  for arcs $c'$, we have that $c_L'$ lies on $P_0$ and $c_R'$ lies on $P_1$, so that the start point and the endpoint of $c$ is mapped
      to the point with the smaller and the larger $z$-value in $B(i) \cap r(B_i)$ respectively.

Let $\tilde{\S}_0 = \tilde{\S}$, $\tilde{\S}_i = r(\tilde{\S}_{i-1})$, for $i =1, \dots, n-1$. Note that the image of
the curve $c_L^{i}$ (the copy of $c_L$ in $\tilde{\S}_i$) coincides with the image of $c_R^{i-1}$ (the copy of $c_R$ in $\tilde{\S}_{i-1}$), for $i \in \mathbb{Z}_n$.
%By our construction the
%directions of the curves agree, and t
This yields a surface $\S'$, invariant under $r$, such that $\S$ is its (branched) covering surface  embedded in $\mathbb{R}^3$.
Note that $\S'$ is connected since $t \ge 1$, and the corresponding branched cover has at least two branch points.

Similarly by adding $h^{-1}(\mathcal{M}(e))$ for $e \in E(G)$ and then taking the orbits of the added points under the group generated by $r$ we get a graph $G'$ embedded in $\S'$, on $n |V(G)|$ vertices and with $n|E(G)|$ edges, such that $\mathcal{M}(G)$ is the quotient space of  $\mathcal{M}'(G')$ with respect to the $n$-fold rotation. To see that $G'$ is isomorphic to the voltage graph derived from $G$, go along an arc $y=\mathcal{M}'(e)$ for $e \in E(G)$, and note that for any $i, s \in \mathbb{Z}_n$, $s\in \{-1,1\}$, the curve  $y$ passes from $\tilde{\S}_i$ to $\tilde{\S}_{(i+s)\bmod n}$ at a point $x$ if and only if $\sigma(x, y, c) = s$.

Now the claim about the genus of $\S'$ follows from the Riemann-Hurwitz formula as above.
\end{proof}

%\begin{lemma}
\begin{prop}\label{prop.rotation}
    Let $\mathcal{M}$ be a piecewise linear embedding of a graph $G$ into a closed connected orientable surface $\S \subset \mathbb{R}^3$ such that there is a face $F_H$ bounded by a Hamiltonian cycle. Suppose $n=|V(G)| \ge 3$ and 
    $\mathcal{M}$ has n-fold rotational symmetry that leaves the boundary of $F_H$ invariant. 
    %Suppose $\mathcal{M}$ is an embedding of $G$ into a closed connected orientable surface $\S \subseteq \mathbb{R}^3$.
    %Let $I$ be a 3-dimensional n-way interchange  $(G,H,$ $\mathcal{M}, \S)$ with $n$-fold rotational symmetry. 
    Then no point in $\mathcal{M}(V(G))$ lies on the rotational symmetry axis $a$.
\end{prop}
%{\color{orange}
\begin{proof}
%
%
%%actually the automorphism maps the whole face to itself.
%
%%By the fixed point theorem, the isometry that maps a compact surface $\S$ to itself, must have a fixed point (which does
%%not necessarily belong to $\S$. The only isometry groups in $\mathbb{R}^3$ that have a fixed point and order at least 3 are rotation about an axis by angle $\frac {2 \pi} / k$, $k \ge 3$ and rotation about an axis by an angle at least $\frac {2 \pi} / k$
%
%%\dots
%%By definition we require that a symmetry group of $I$ fixes a face $f_H$ bounded by a Hamiltonian cycle.
%%$n$-fold rotation group fixes exactly one point in the interior of $f_H$
%
%
%
%%\dots
%
%%Let us now show that no point $x \in \mathcal{M}(G)$ lies on the rotation axis $a$.
%%Fix arbitrary orientations of the edges of $G$ (so that we can identify their images with curves).
%
If $x=\mathcal{M}(v)$ is a fixed point, then consider the orbit of the curve $c_e = \mathcal{M}(e)$ such that $e$ is on 
the boundary of $F_H$.
%the face $f_H$ bounded by the Hamiltonian cycle $H$ whose boundary is invariant under the rotation.
 We get that each of $n$ edges incident to $v$ belongs to the boundary of $H$. This is a contradiction, since $H$ contains $v$ only once.
\end{proof}
%}

\subsection{Proof of the lower bound: the combinatorial part}
\label{sec.geom.lower2}

Recall that a \emph{region} of an embedding $\mathcal{M}$ of $G$ to $\S$ consists of all elements
of an equivalence class of points that can be connected by a simple arc in
$\S \setminus \mathcal{M}(G)$. A face is a region that is homeomorphic to an open disk in $\mathbb{R}^2$.
We call the sum of lengths of all boundary walks of a region $R$ the (boundary) \emph{size} of $R$.
We call $R$ a \emph{$k$-region}, if its size is $k$. 
\begin{lemma}\label{lem.geometric_lower}
Let $n\ge 3$ be an integer. %satisfy $n \bmod 4 \in \{0,3\}$.
 Let $\mathcal{M}$ be a piecewise linear embedding of $D_n$ into a closed
connected oriented surface $\S \subset \mathbb{R}^3$ of genus $g$. Let $C$ be a finite set of pairwise disjoint simple curves in $\S$, $t$ of which are arcs, and $|C|-t$ of which are closed curves. 
Suppose $t \ge 1$ and the voltage function $\alpha_{\mathcal{M}, C, n}: E(D_n) \to \mathbb{Z}_n$ (\ref{eq.cutvolt}) is well defined and  bijective. 

Then $n g + (n-1) (t-1) \ge \Lgeom(n)$. 
\end{lemma}

\begin{proof}
Let $\alpha=\alpha_{\mathcal{M}, C, n}$.
Suppose the claim is false, and take $\mathcal{M}$ and $C$ such that $t \ge 1$, $\alpha$ is bijective and
%\begin{align}\label{eq.contr}
$n g + (n-1) (t-1) < \Lgeom(n)$. 
%\end{align}
Then since $(n-1)(t-1) \ge 0$ and $t,g,n$ and $\Lgeom(n)$ are integers %by Proposition~\ref{prop.cuts}
%checked 
\begin{align} \label{eq.g.assumed}
g \le \left \lfloor \frac {\Lgeom(n) -1} n \right\rfloor =  \begin{cases} 
\frac {n-4} 4 &\mbox{for }  n\equiv0\,(\bmod \, 4), \\
\frac {n-5} 4 &\mbox{for }  n\equiv1\,(\bmod \, 4), \\
\frac {n-6} 4 &\mbox{for }  n\equiv2\,(\bmod \, 4), \\
\frac {n-3} 4 &\mbox{for } n\equiv3\,(\bmod \, 4).
\end{cases}
\end{align}
%For $n=3$ we immediately have a contradiction, since genus must be non-negative.
%So we further assume $n \ge 4$.

%We first bound the number of 2-faces of the embedding from below. Suppose the embedding has
Since $D_n$ is bipartite, each region of $\mathcal{M}$ must have an even size.
Let $k$ be the total number of 2-regions.
$D_n$ has $n$ edges and $2$ vertices. It follows by Euler's formula that the number $r$ of regions is
\begin{align*}
   r \ge 2- 2g - 2 + n =n-2g.
\end{align*}
Assume first $k < n - 2g$. Each edge either belongs to the boundary walks of two regions or to the boundary walks of the same region twice. Thus, the average size of 
those regions that have size at least 4 is
%proof for non-2-cell embeddings breaks down here
\begin{align*}
\frac {2n - 2k} {r - k} \le \frac {2n - 2k} {n- 2g - k}. 
\end{align*}
On the other hand, this average must be at least 4. Thus
\begin{align}
2n-2k \ge 4 (n- 2g - k) \nonumber \\
2k \ge 2n - 8g \nonumber \\
k \ge n - 4g. \label{eq.k.lower}
\end{align}
The last inequality trivially holds also when $k \ge n-2g$.

By our initial assumption, the number of arcs in $C$ satisfies
\begin{align} \label{eq.t.upper}
 t<1 + \frac{\Lgeom(n) - n g} {n-1}.
\end{align}

Let $R$ be a 2-region with boundary consisting of arcs $b=\mathcal{M}(e_1)$ and $c=\mathcal{M}(e_2)$
from the white vertex to the black vertex. $e_1$ and $e_2$ must be different edges, since $n > 1$.
%{\color{orange} 
We use the following simple property: if a curve $a \in C$ does not have an endpoint in $R$, then $\sigma(\mathcal{M}, b, a)=\sigma(\mathcal{M}, c,a)$.
If $a$ has just one endpoint in $R$, then $|\sigma(\mathcal{M}, b, a) - \sigma(\mathcal{M}, c,a)| = 1$. This property is a simple consequence of the definition of a region, the definition of $\sigma$ and the fact that $\alpha$ is well-defined (recall that we work in the piecewise-linear setting). 
%}

Thus if no arc in $C$ has an endpoint in $R$, then  $\alpha(e_1) = \alpha(e_2)$, which contradicts the assumption that
$\alpha$ is bijective. This implies that each 2-region must have at least one of $2t$ total endpoints.
Since regions are pairwise disjoint sets, by the pigeonhole principle we must have $2t \ge k$.

In the case $n \equiv 3\, (\bmod\, 4)$, suppose $g = (n-3)/4$. Then (\ref{eq.t.upper}) implies
$t=1$, so there are just two endpoints, but by (\ref{eq.k.lower}) the number of 2-regions is at least $3$,
a contradiction. Thus we can further assume
\begin{align} \label{eq.g.assumed2}
g \le %\le \begin{cases} 
\frac {n-3} 4 - 1 & &\mbox{for } n\equiv3\,(\bmod \, 4).
%\end{cases}
\end{align}
From (\ref{eq.k.lower}) and  (\ref{eq.t.upper}) we have
\begin{align*}
n - 4g \le k \le 2t <  2 + \frac{2\Lgeom(n) - 2n g} {n-1}
\end{align*}
or 
\begin{align*}
&g > \frac{(n-2)(n-1) - 2\Lgeom(n)} {2n-4} = 
%\begin{cases}
%\frac n 4 - \frac{n-1} {n-2}, &\mbox{for } n = 0\, (\bmod\, 4), \\
%\frac {n^2 - 7n + 8} {4 (n-2)}, &\mbox{for } n= 3\,(\bmod\, 4).
%\end{cases} \\
%&=
\begin{cases}
\frac {n-4} 4, &\mbox{for } n = 0\, (\bmod\, 4), \\
\frac {n-5} 4 + \frac 1 2 - \frac 1 {2 (n-2)}, &\mbox{for } n = 1\, (\bmod\, 4), \\
\frac {n-6} 4 + 1, &\mbox{for } n = 2\, (\bmod\, 4), \\
\frac {n-3} 4 - \frac 1 2 - \frac 1 {2(n-2)}, &\mbox{for } n= 3\,(\bmod\, 4).
\end{cases} \\
\end{align*} %checked, needs no more checking
In all cases we have a contradiction to (\ref{eq.g.assumed}) and (\ref{eq.g.assumed2}). 
%In case $n= 3\,(\bmod\, 4)$ this is a contradiction to (\ref{eq.g.assumed}) and the proof is complete.
%
%For $n= 0\,(\bmod\, 4)$, the above bound together with (\ref{eq.g.assumed}) implies
%that $g= \frac n 4 - 1$. Thus by (\ref{eq.k.lower}) and (\ref{eq.t.upper}) $k \ge 4$ and 
%$t < 2 + \frac 1 {n-1}$. If $t =1$ or $k > 4$ and $t=2$ then at least one 2-face contains no endpoints,
%hence it must be $t=2$ and $k=4$, so that there are four faces containing one endpoint each.
\end{proof}

\subsection{Proof of the upper bound: optimal ring road constructions}
\label{sec.geom.upper}

We represent rotations of an embedding of a directed multigraph as cyclic sequences of symbols $e_{i_1}^{s_1} \dots e_{i_k}^{s_k}$. We have $s_j=-1$ if the edge $e_{i_j}$ is incoming and $s_j=1$ otherwise, also we shorten $e_i = e_i^1$. Similarly we denote faces of embedded directed multigraphs as cyclic permutations of the same form. In this case the symbol $e^{-1}$ indicates that the corresponding
arc $\mathcal{M}(e)$ in the counter-clockwise facial walk is traversed in the opposite direction.

 For a degree 2 vertex $x$ in a multigraph
\emph{smoothing} $x$ means connecting the neighbours of $x$ by an edge and removing $x$; this is extended to embedded or directed multigraphs in the obvious way.

\begin{lemma} \label{lem.geomupper}
    For each integer $n \ge 2$ there is a 3-dimensional piecewise linear complete interchange
    that has $n$-fold rotational symmetry and genus at most $\Lgeom(n)$.
\end{lemma}
%\bigskip

\begin{proof}
    Consider a compact piecewise-linear orientable surface $\B=\B(n)$ embedded in $\mathbb{R}^3$, obtained
    by adding some number $g$, $g \ge 0$ handles to a rectangle with boundary $abcd$. Place a black point $v_b$ and a white point $v_w$ on the boundary segment $da$, so that $d$, $v_w$, $v_b$, $a$
    are all distinct and occur in this order when going from 
    %$a$ to $d$.
   $d$ to $a$.
   
Take a permutation $\pi = \pi(n)= (\pi_1, \dots, \pi_{n-1})$ of $\{1, \dots, n-1\}$ and place points $\pi_{n-1}, \dots, \pi_1$ on the (interior of the) segment $cd$ in this order. Similarly, take a permutation $\pi' = \pi'(n)= (\pi_1', \dots, \pi_{n-1}')$ of $\{0, \dots, n-2\}$ which is related to the permutation $\pi$ through 
\begin{align}\label{eq.pi_prime}
&\pi_i' = \pi_i - 1  &&\mbox{for } i \in \{1, n-1\}.
\end{align}
Place points $\pi_1', \dots, \pi_{n-1}'$ on the segment $ab$ in this order\footnote{In order to keep the notation simpler, we use symbols $\pi_1, \dots, \pi_n$ and $\pi_1', \dots, \pi_n'$ to refer both to two disjoint sets of points and to orderings of two overlapping sequences. The precise meaning should be interpreted from the context.}.

The point $\pi_i$, if $\pi_i = k$, will model $n-k$ lanes incoming from motorways $K, K-1, \dots, K-(n-k)+1$ and exiting to motorway $K + k$. Here we think of motorways as indexed by $\mathbb{Z}_n$ and arranged in the 
counter-clockwise order, so that $K$ is (the index of) the current motorway, $K+1$ is the motorway to its right, etc.  
%\m{Introduce notation for the inverse of $\pi$?}

Suppose there exists an embedding $\mathcal{M}=\mathcal{M}(n)$ of a graph $G=G(n)$ with vertices mapped to $\pi_1, \dots, \pi_{n}$, $\pi_1', \dots, \pi_{n-1}'$, $v_w,v_b$ and edges mapped to the following arcs in $\B$:
\begin{enumerate}
\item[(a)] $\pi_t' v_b$, where $\pi_t' = 0$ (i.e., the unique group of $n-1$ lanes exiting at the current motorway $K$),
\item[(b)] $v_w \pi_i$, for each $i \in \{1, \dots, n-1\}$ (i.e., the lanes from the current motorway $K$ to each of the $n-1$ groups of lanes on the ring road with different destinations),
\item[(c)] for every $k \in \{1, \dots, n-2\}$ the unique arc $\pi_i' \pi_j$ such that $\pi_i' = \pi_j = k$ (i.e. the group of lanes
from motorways $K-1, \dots, K - (n-k) + 1$ that continue on the ring road passing the current motorway $K$). 
\end{enumerate}

Expand each arc $\pi_i' \pi_j$ defined in (c), if, say, $\pi_i' = \pi_j = k$ into $n-k-1$ disjoint `parallel' arcs. Also expand the point $\pi_i'$ into $n-k-1$ points, the starting points of these arcs, also placed on the boundary segment $ab$. By our definition, $\pi_i = \pi_i'+1= k + 1$. Replace the point $\pi_i$  by $n-k-1$ new points on the segment $cd$ (so that $n - (k+1) - 1$ of them are
endpoints of the parallel arcs, and one point is the endpoint of the arc $v_w \pi_i$), so that the order of these points corresponds to the rotation at $\pi_i$. %\m{clarify}.
 Next, add an edge $v_w v_b$ with its image subset of the segment $da$.

We get an embedding $\mathcal{M}'$ of a new (expanded) graph $G'$ into $\B$ with degree 1 vertices corresponding to $n-1 + \dots + 1 = n (n-1)/2$ points on each of the boundary segments $ab$ and $cd$, 2 vertices of degree $n$ corresponding to $v_w$ and $v_b$, and the arcs consisting of the union of $v_w v_b$ and the arcs (a), (b) and (c) replaced using the above modification.

Now make $n$ copies of $(\B, \mathcal{M}', G')$, $B_0'$, \dots, $B_{n-1}'$ and 
%{\color{orange} 
arrange them (using rotations and translations) counterclockwise into a `ring',
%}
so that each block is disjoint and we have $n$-fold rotational symmetry. For $i\in \mathbb{Z}_n$ identify the segment $b^{(i+1)}a^{(i+1)}$ of block $B_{i+1}'$ and the segment
$d^{(i)} c^{(i)}$ of $B_{i}'$ with the corresponding sides of a new polygon $P_i = a^{(i+1)} b^{(i+1)} c^{(i)} d^{(i)}$; also connect the $k$-th point on $d^{(i)} c^{(i)}$ with the $k$-th point of $a^{(i)} b^{(i)}$, using a straight arc inside $P_i$  for each $k \in \{1, \dots, n(n-1)/2\}$. 
%{\color{orange} 
Make the surface closed by gluing
%}
 a genus 0 surface with boundary $b^{(0)}c^{(0)} \dots b^{(n-1)} c^{(n-1)}$ and a genus 0 surface with boundary $d^{(n-1)}a^{(n-1)} \dots d^{(0)} a^{(0)}$. It easy to ensure the n-fold symmetry is preserved and newly added faces are piecewise linear. 

For any $K, k \in \mathbb{Z}_n$ and $k \ne 0$ follow the path $P_{K,k}$ from $v_w^{(K)}$ in $B_K'$ such that the first edge on this path corresponds to $v_w \pi_i$ with $\pi_i = k$ and ending at the first copy of $v_b$, say $v_b^{(j)}$ in $B_j'$. 
%For $k=0$, we can define the path $P_{K,k}$ to be just $v_{w}^{(K)}v_{b}^{(K)}$.
 Each internal vertex on $P_{K,k}$ has degree two. Using
(\ref{eq.pi_prime}) and simple induction we see that $j \equiv K+k \pmod n$.
We get a 3-dimensional embedding of a subdivision of $K_{n,n}$ with $n$-fold rotation symmetry into a surface of genus $n g$.
To ensure the underlying embedding of $K_{n,n}$ has a face bounded by a Hamiltonian cycle $v_w^{(n-1)} v_b^{(n-1)} \dots v_w^{(0)} v_b^{(0)}$ that contains a fixed point, for each $i \in \mathbb{Z}_n$ remove the path connecting $v_w^{(i)}$ with
$v_b^{(i+1)}$ and add a new arc $v_w^{(i)} v_b^{(i+1)}$ that coincides with parts of the boundary segments $a^{(i)} d^{(i)}$ and  $a^{(i+1)} d^{(i+1)}$  of $B_i'$ and $B_{i+1}$ respectively.

%{\color{orange} 
Smooth
%} 
all the degree 2 vertices in the resulting graph that are not copies of $v_w$ or $v_b$. We have obtained
an embedding of $K_{n,n}$ with an $n$-fold rotation symmetry into a surface of genus $gn$. We call this embedding the \emph{ring road embedding}
derived from $(\B, G, \mathcal{M})$.
\m{Use Proposition~\ref{prop.cuts}?}

To finish the proof for $n \bmod 4 \in \{1,2\}$ it suffices to show that the surface $\B(n)$ of genus $g=g(n)$,
the graph $G(n)$ and its embedding $\mathcal{M}(n)$ into $\B(n)$ exist with
\begin{align}\label{eq.g_block_upper}
g \le \frac{\Lgeom(n)} n  = 
\begin{cases} 
\frac{n-1} 4 &\mbox{for } n \equiv 1\,(\bmod\,4); \\
\frac{n-2} 4 &\mbox{for } n \equiv 2\,(\bmod\,4).
\end{cases}
\end{align}
%We will then modify the constructions for $n \equiv 2\pmod4$ to prove the remaining cases $n \bmod 4 \in \{0,3\}$.

\begin{figure}
\centering
\includegraphics[width=3cm]{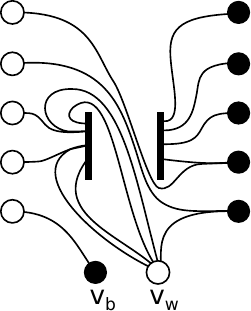}
\caption{The solution for $n=6$.}
\label{fig.block6}
\end{figure}

\emph{Case $n \equiv 2\pmod4$.}\quad The construction for $n=2$ is trivial, so we assume $n \ge 6$. 
The building block of our construction is a (the?) solution for $n = 6$ with $\pi = \pi^{(6)}= (1,4,3,2,5)$, $\pi'=\pi^{(6)'}=(0,3,2,1,4)$,
and the rotations\footnote{Note that for points on the boundary, rotations are not identical modulo cyclic shifts.}
\begin{align*}
%\label{eq.pi_6}
 &\pi_1: v_w \pi_4'  && \pi_2: \pi_5' v_w && \pi_3: v_w \pi_2' &&\pi_4: \pi_3' v_w \nonumber \\
 &\pi_2': \pi_3     && \pi_3': \pi_4 && \pi_4': \pi_1      && \pi_5': \pi_2 \nonumber   && \\ 
 %&\mbox{and} \\
 &\pi_5: v_w && \pi_1': v_b &&v_w: \pi_3\pi_2\pi_5\pi_4\pi_1 && v_b: \pi_1'. 
\end{align*}
This embedding has genus $(6-2)/4 = 1$ as required. It is shown in Figure~\ref{fig.block6}.
We construct a solution for arbitrary $n$,  $n \equiv 2\pmod 4$ by concatenating $(n-2)/4$ copies of the solution for $n=6$, see
Figure~\ref{fig.ring_road_block}, right. 

The concatenation is defined as follows. We use permutations $\pi = \pi^{(n)}$  and $\pi' = \pi'^{(n)}$ as in (\ref{eq.pi_prime}),
where 
\begin{align}
\label{eq.pi_gen_2}
& \pi_{4t+1} = 4t+1, && \pi_{4t+2} = 4t+4, && \pi_{4t+3} = 4t+3,\nonumber \\
&\pi_{4t+4} = 4t+2 && \mbox{for }t = 0, \dots, \left\lfloor n/4 \right \rfloor - 1 \mbox{ and}  \\
&\pi_{n-1} = n-1.\nonumber
 \end{align}

\begin{figure}
\centering
\begin{subfigure}[b]{0.45\textwidth}
\centering
\includegraphics[width=0.7\linewidth]{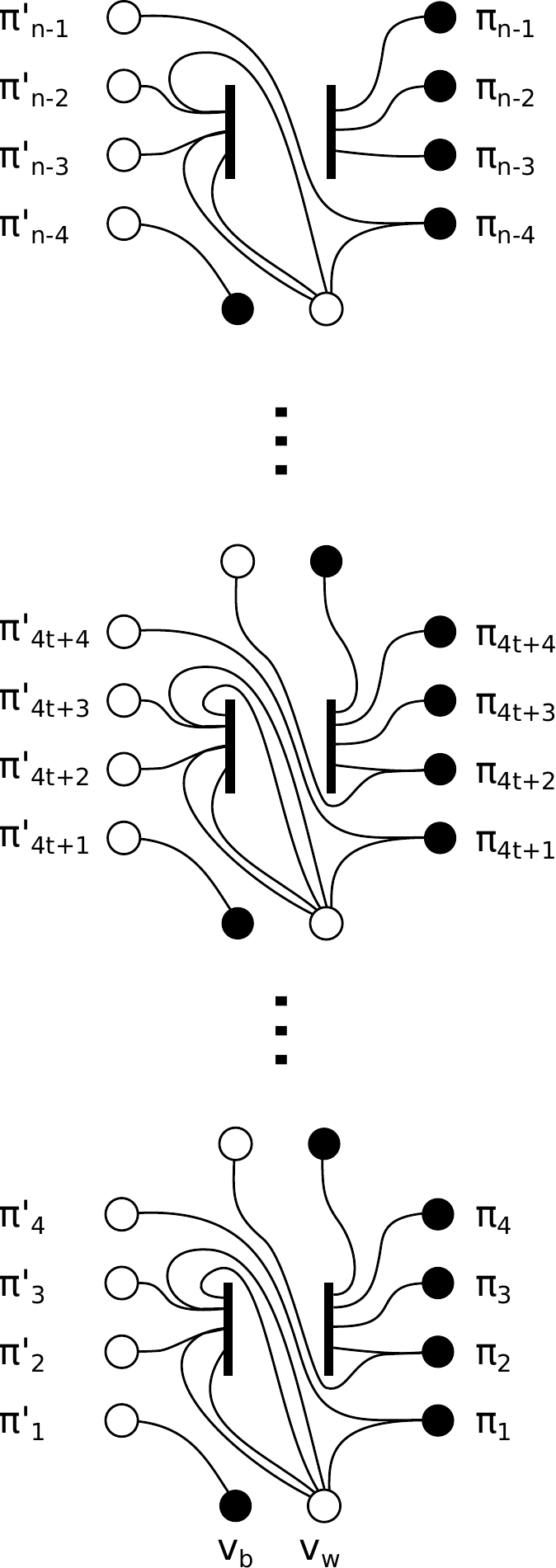}
\end{subfigure}
\begin{subfigure}[b]{0.45\textwidth}
\centering
\includegraphics[width=0.7\linewidth]{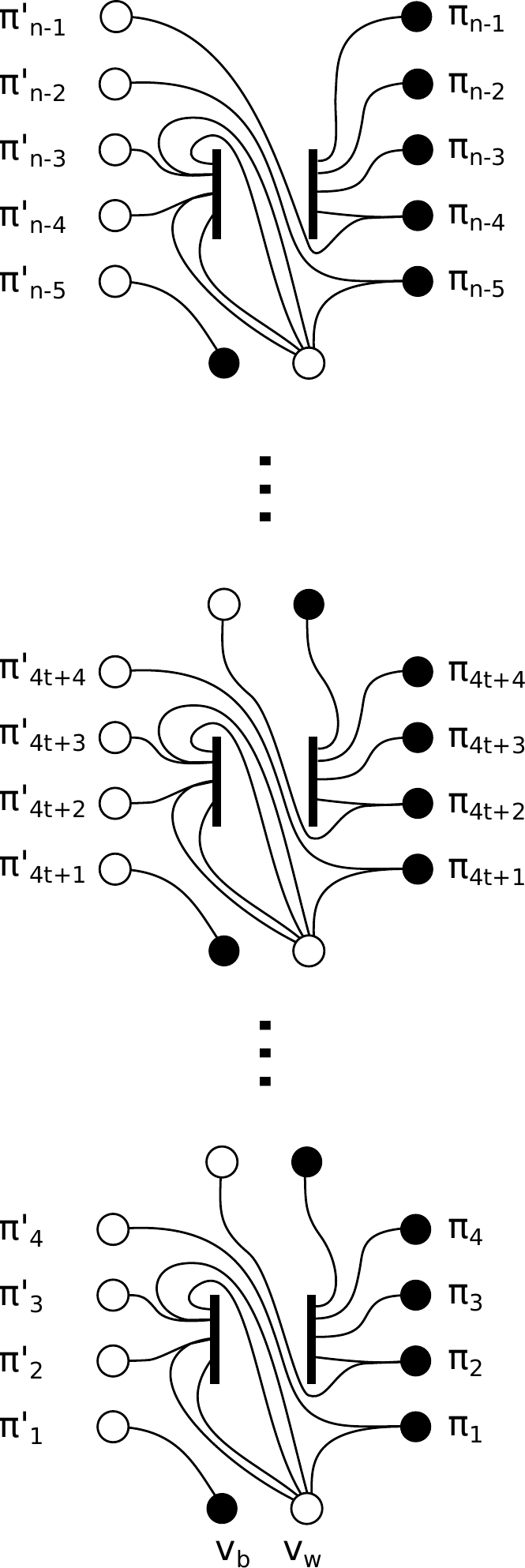}
\end{subfigure}
\caption{\label{fig.ring_road_block}Construction of a block $\B=\B(n)$ from block $\B(6)$, $n\equiv1\pmod 4$ (left) and $n\equiv2\pmod4$ (right).}
\end{figure}
We aim to construct an embedding with the following rotations for $t \in \{0, \dots, (n-2)/4 - 1\}$:
\begin{align}
 &\pi_{4t+1}: v_w \pi_{4t+4}'  && \pi_{4t+2}: \pi_{4(t+1)+1}' v_w && \pi_{4t+3}: v_w \pi_{4t+2}' && \pi_{4t+4}: \pi_{4t+3}' v_w \nonumber \\
 &\pi_{4t+2}': \pi_{4t+3}     && \pi_{4t+3}': \pi_{4t+4} && \pi_{4t+4}': \pi_{4t+1}      && \pi_{4(t+1)+1}': \pi_{4t+2} \nonumber \\
 & \pi_1': v_b && v_b: \pi_1' \nonumber \\
 &\mbox{and} && \pi_{n-1}: v_w.  \label{eq.rot_block_string}
\end{align}
We claim that for each $n$ such that $n=4k+2$,  $\B(n)$ and $\mathcal{M}(n)$ can be chosen so that (\ref{eq.g_block_upper}), (\ref{eq.pi_gen_2}) and (\ref{eq.rot_block_string}) are satisfied. To prove this, we use induction on $k$. 
We have already shown the claim for $k=1$ with embedding $\mathcal{M}(6)$ into a surface $\B(6)$ of genus 1 with boundary rectangle $abcd$. Let $n=4k+2$ with $k \ge 2$ and suppose the claim is true for $n=4(k-1) + 2$, i.e., there is a required embedding $\mathcal{M}(n-4)$ into a surface $\B(n-4)$ with boundary rectangle $a'b'c'd'$ with permutations $\pi = \pi^{(n-4)}$, $\pi' = \pi'^{(n-4)}$ as in (\ref{eq.pi_gen_2}), rotations (\ref{eq.rot_block_string}), rotation at $v_w= v_w^{(n-4)}$:
 \[
 v_w: \pi_{\sigma_1}^{(n-4)} \dots \pi_{\sigma_{n-4-1}}^{(n-4)}
 \]
 and genus at most $\Lgeom(n-4)/(n-4) = \Lgeom(n)/n - 1$. Now consider $\pi^{(n-4)}$ and $\pi'^{(n-4)}$ as permutations of $(5, \dots, n-1)$ and $(4, \dots, n-2)$ respectively, i.e., simply shift the sets permuted by $\pi^{(n-4)}$ and $\pi'^{(n-4)}$ by~4.
 
 Identify $bc$ with $a'd'$ and identify the points (and corresponding vertices) $v_w^{(n-4)}$
with $\pi^{(6)}_5$ and $v_b^{(n-4)}$  with $\pi'^{(6)}_5$. Here $v_w^{(t)}$ and $v_b^{(t)}$ denote the points $v_w$ and $v_b$ respectively in the corresponding embedding $\mathcal{M}(t)$. Contract the edge mapped to $\pi'^{(6)}_5 \pi_2^{(6)}$ and the edge mapped to $\pi^{(6)}_5 v_w^{n-4}$ (
%{\color{orange} 
smoothing the point $\pi'^{(6)}_5$ and pulling the point $\pi^{(6)}_5$ to $v_b$
%}
). We get a construction with boundary
rectangle $ab'c'd$, the permutation 
\[
\tilde{\pi} = (\pi_1^{(6)}, \dots, \pi_4^{(6)}, \pi_1^{(n-4)}, \dots, \pi_{n-4-1}^{(n-4)})
\] on the segment $dc'$,
and the permutation
\[
\tilde{\pi}' = (\pi'^{(6)}_1, \dots, \pi'^{(6)}_4, \pi'^{(n-4)}_1, \dots, \pi'^{(n-4)}_{n-4-1})
\] on the segment $ab'$
and the following rotation at $v_w$:
\begin{equation}\label{eq.rot_v_w}
v_w: \pi_3^{(6)}\pi_2^{(6)} \pi_{\sigma_1}^{(n-4)} \dots \pi_{\sigma_{n-4-1}}^{(n-4)}  \pi_4^{(6)}\pi_1^{(6)}. 
\end{equation}
The rotation system of the resulting embedding $\tilde{M}$ satisfies (\ref{eq.pi_gen_2}) and (\ref{eq.rot_block_string}) with $\pi = \tilde{\pi}$
and $\pi'=\tilde{\pi}'$. Finally, the genus of $\tilde{M}$ is at most $1 + \Lgeom(n)/n -1 = \Lgeom(n)/n$.

\emph{Case $n \equiv 1\pmod4$.}\quad We define $\pi=\pi^{(n)}$  as in (\ref{eq.pi_gen_2}), except we drop the last equality $\pi_{n-1}=n-1$. We define the rotations for $\pi_1, \dots, \pi_{n-1}$, $\pi_1', \dots, \pi_{n-1}'$ and $v_b$ as in (\ref{eq.rot_block_string}), except
we drop the last equality, and for $t = (n-1)/4 - 1$, we replace the rotation at $\pi_{4 t + 2}=\pi_{n-3}$ with:
\[
    \pi_{4t + 2}: v_w.
\]
The rest of the proof is analogous to the proof in the case $n \equiv 2\pmod4$, so we omit it. Figure~\ref{fig.ring_road_block}, left, illustrates the construction.

\textit{Case $n \equiv 3 \pmod 4$.}\quad For $n=3$ we use Proposition~\ref{prop.cuts} with the specific embedding $\mathcal{M}_1(3)$ of $D_3$ into a closed surface $\S_1(3)$ of genus 0 shown in Figure~\ref{fig.sym3}.
\begin{figure}
\centering
\begin{subfigure}[b]{0.45\textwidth}
\centering
\includegraphics[width=0.4\linewidth]{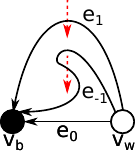}
\caption{}
\end{subfigure}
\begin{subfigure}[b]{0.45\textwidth}
\centering
\includegraphics[width=0.6\linewidth]{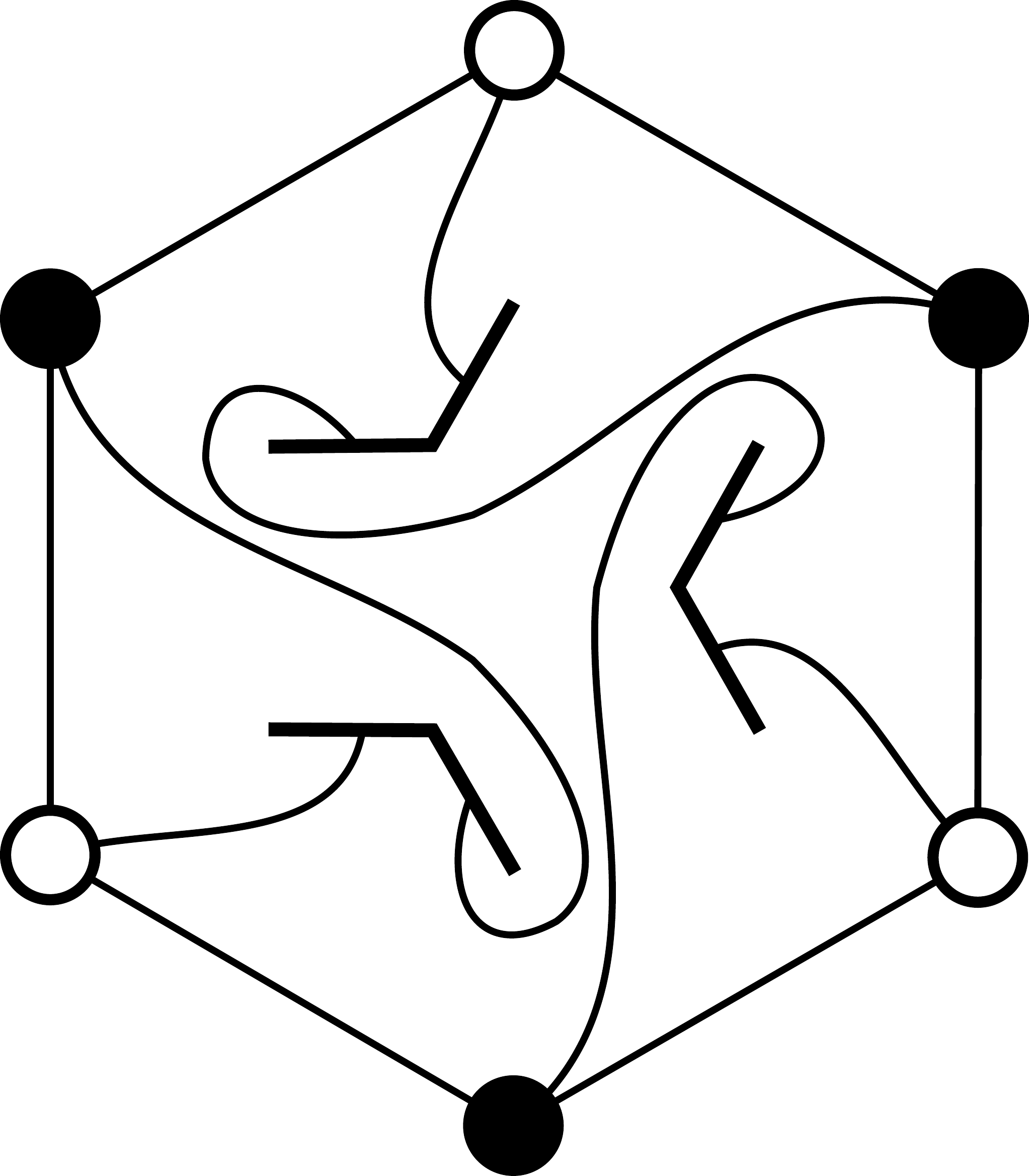}
\caption{}
\end{subfigure}
\caption{\label{fig.sym3} (a) The embedding $\mathcal{M}_1(3)$ of $D_3$ and a set $C$ of two cut arcs (dashed), such that $\alpha_{\mathcal{M}_1,C, 3}: E(D_3) \to \mathbb{Z}_3$ is bijective. (b) The resulting 3-fold symmetric embedding of $K_{3,3}$ with a face bounded by a Hamiltonian cycle.}
\end{figure}
For  $n \ge 7$, start with the embedding $\mathcal{M}(n-1)$ given above. Expand it into
the 
 embedding $(\B, \mathcal{M}', G')$ with boundary $abcd$ as described in the first part of the proof. Identify the segment $ab$ with the segment $dc$. Also identify the $k$th point on $ab$ with the $k$th point of $dc$ for $k \in \{1, \dots, (n-1)(n-2)/2\}$, and smooth each of the identified points. 
 %{\color{orange}
 Cap the two holes with boundaries $cb$ and $da$ respectively by gluing a new 1-face to each of them. We obtain an embedding $\mathcal{M}_1=\mathcal{M}_1(n-1)$ of a copy $G_1$ of $D_{n-1}$ into the surface $\S_1 = \S_1(n-1) \subset \mathbb{R}^3$ of genus $g \le \Lgeom(n-1)/(n-1) = (n-3)/4$. It is easy to ensure that $\mathcal{M}_1$ is piecewise-linear: for example, if we bend the rectangle $abcd$ into a cylinder and then take its piecewise-linear approximation. Let $C_1 = \{X_1\}$ where $X_1$ is a cut (an arc in $\S_1$) that coincides with the segment $dc$.
%}. 
 If we apply Proposition~\ref{prop.cuts} with $\mathcal{S}_1$, $\mathcal{M}_1$ and the group $\mathbb{Z}_{n-1}$, we obtain the solution for $n-1 \equiv 2\,(\bmod\,4)$ roads that we described above. Let us instead apply Proposition~\ref{prop.cuts} with the group $\mathbb{Z}_n$.
 The corresponding voltage function $\alpha_{\mathcal{M}_1, C_1, n}:E(G_1) \to \mathbb{Z}_{n}$ is injective; indeed 
 by our construction the outgoing arc $e_k$ that corresponds to $v_w \pi_i$,  $\pi_i = k$ crosses $X_1$ exactly $k$ times, each time with sign 1, so $\sigma(e_k, C_1) = k$. Thus the image of $\alpha_{\mathcal{M}_1, C_1, n}$ is $\{0, \dots, n-2\}$.

\m{Tidy up the notation for faces and rotations for directed graphs.}
Insert a new edge $e_{-1}$ to $G_1$ and map it to an arc from $v_w$ to $v_b$ parallel to $\mathcal{M}_1(e_0)$, so that the rotation at $v_w$ is $\dots e_0 e_{-1} \dots$ and the rotation
at $v_b$ is $ \dots e_{-1}^{-1} e_0^{-1} \dots$. Now add a new cut $X_2$ disjoint from the image of $G_1$,
that has only one point in common with $e_{-1}$ and $\sigma(e_{-1}, X_2) = -1$. We obtained an embedding $\tilde{\mathcal{M}}$ of $D_n$ into the surface $\S_1$ of genus $g \le (n-3)/4$ and $t=2$ cuts, so that $\alpha_{\tilde{\mathcal{M}}_1, \{X_1, X_2\}, n}: E(D_n) \to \mathbb{Z}_n$ is bijective. Now the proof follows by Proposition~\ref{prop.cuts}.

\textit{Case $n \equiv 0 \pmod 4$.} The solution for $n=4$ is exceptional, so we assume $n \ge 8$. We first determine the rotation at $v_b$ in the embedding $\mathcal{M}_1(n-2)$ of the copy $G_1$ of $D_{n-2}$ into the `quotient' surface $\S_1(n-2)$. %\m{Padalinti į daugiau lemų?}
\begin{claim}\label{claim.rot_v_b}
   For $n \equiv 2\pmod 4$ in the embedding $\mathcal{M}_1(n)$ the rotation at $v_b$ is \m{curves and edges: same notation..}
   \begin{align}
      v_b: e_1^{-1} e_3^{-1} \dots e_{n-3}^{-1} e_{n-1}^{-1} e_{n-2}^{-1}e_{n-4}^{-1} \dots e_2^{-1} e_0^{-1}; \label{eq.rot_v_b}
   \end{align}
   Furthermore, $e_0^{-1}e_3e_5^{-1}e_2$ and $e_2^{-1}e_1e_3^{-1}e_4$ are among the faces of $\mathcal{M}_1(n)$.
\end{claim}
\begin{proof}
    Let  $\B(n), G(n), \mathcal{M}(n)$ be obtained with $\pi = \pi^{(n)}$, $\pi' = \pi'^{(n)}$ as above.
     We say that $k \in \{0, \dots, n-2\}$ is \emph{joined from the left} on the ring road if either $k=0$ or in $\mathcal{M}(n)$ the rotation at $\pi_i$ where $\pi_i = k$ is $\pi_i: \pi_j' v_w$.
    Here $j$ is such that $\pi_j' = k$. Otherwise we say that $k$ is \emph{joined from the right}.

    %\m{nepakankamai formalu}
    %{\color{orange} 
    Consider travelling from the motorway 0 to the motorway $n-1$ in the ring road embedding derived from $\B(n)$, $G(n)$ and  $\mathcal{M}(n)$.
    % }
     At block $0$ we exit at the group of lanes corresponding to $\pi_j^{(0)}$ with $\pi_j = n-1$ (i.e. $j=n-1$).
    Next, we enter block $1$ at the point $\pi'^{(1)}_{n-1} = n-2$ and exit at $\pi^{(1)}_j$ such that $\pi_j=n-2$. We exit block $2$ at $\pi^{(2)}_j$ such that $\pi_j=n-3$, and so on.
     It follows that the rotation at $v_b^{(n-1)}$ in the
   ring road embedding is $r_{n-1}$ where $r_i$, $i \in \{0, \dots, n-1\}$ are defined by
    \begin{align*}
        r_i = \begin{cases}
        v_0, &\mbox{ if } i=0; \\
        r_{i-1} v_i, &\mbox{ if $i \ge 1$ and  $n-i-1$ is joined from the left};\\
        v_i r_{i-1}, &\mbox{ if $i \ge 1$ and  $n-i-1$ is joined from the right}.
        \end{cases}
    \end{align*}
    Here for a sequence $y=(y_1, \dots, y_k)$ and an element $x$ we use notation $yx = (y_1, \dots, y_k, x)$ and
    $xy = (x, y_1, \dots, y_k)$.
    Using (\ref{eq.rot_block_string}) we see that $n-i-1$ is joined from the right if $i \equiv 0 \pmod 2$ and joined from the left otherwise.
    Thus the rotation at $v_b^{(n-1)}$ in the ring road embedding is
    \[
    v_b^{(n-1)}: v_w^{(n-2)} v_w^{(n-4)} \dots v_w^{(4)} v_w^{(2)} v_w^{(0)} v_w^{(1)} v_w^{(3)} \dots v_w^{(n-1)}.
    \]
    Since $v_w^{(i)} v_b^{(n-1)}$ in the ring road embedding 
    %{\color{orange}
    is mapped by the covering projection
    %} 
    to the arc $e_{n-1-i}$ in $\mathcal{M}_1$, it follows that the rotation at $v_b$ in $\mathcal{M}_1$ is (\ref{eq.rot_v_b}).
By (\ref{eq.pi_gen_2}), (\ref{eq.rot_v_w}) and our construction, the rotation at $v_w$ in $\mathcal{M}_1$ is
\[
v_w: e_0 e_3 e_4 \dots e_5 e_2 e_1.
\]
Thus using (\ref{eq.rot_v_b}), $\mathcal{M}_1$ has a face $e_0^{-1}e_3e_5^{-1}e_2$ and a face $e_2^{-1}e_1e_3^{-1}e_4$. 
\end{proof}

\medskip

\begin{figure}
\centering
\begin{subfigure}[b]{0.45\textwidth}
\centering
\includegraphics[width=0.7\linewidth]{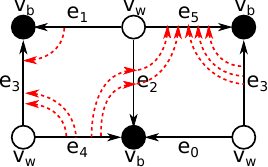}
\caption{}
\end{subfigure}
\begin{subfigure}[b]{0.45\textwidth}
\centering
\includegraphics[width=0.7\linewidth]{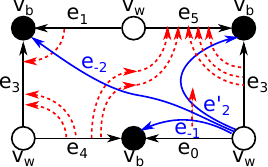}
\caption{}
\end{subfigure}
\caption{\label{fig.modification_4k}Left: faces $e_0^{-1}e_3e_5^{-1}e_2$ and $e_2^{-1}e_1e_3^{-1}e_4$ in the embedding $\mathcal{M}_1$
and the intersection of the cut arc $X_1$ with these faces (red dashes).
Right: adding a new cut arc, replacing the arc $e_2$ and adding new arcs $e_{-1}$ and $e_{-2}$ to solve the case $n \equiv 0\pmod 4$.}
\end{figure}
Let $X_1$ be the cut arc in $\S_1=\S_1(n-2)$ defined above. Using Claim~\ref{claim.rot_v_b}, the 
fact that each crossing of an arc $e_k$ and $X_1$ has sign 1 
%{\color{orange}
 and a simple geometric argument (Jordan-Schoeneflies theorem in the piecewise-linear case), we obtain the order $X_1$ enters and leaves
the faces  $e_0^{-1}e_3e_5^{-1}e_2$ and $e_2^{-1}e_1e_3^{-1}e_4$. The faces and the segments of $X_1$ intersecting them are 
shown, up to homeomorphism in Figure~\ref{fig.modification_4k}(a).  
%}
 
Now in the face $e_0^{-1}e_3e_5^{-1}e_2$ add a new arc $e_{-1}$ from $v_w$ to $v_b$ parallel to $e_0$
and a new arc $e_2'$ parallel to $e_3$. Introduce a new cut arc $X_2$ that starts inside
the newly created face $e_{-1}e_0^{-1}$, ends at the new face $e_3 e_2'^{-1}$ such that 
$X_2$ shares exactly one point with each of $e_2'$ and $e_{-1}$, $\sigma(e_2', X_2) = \sigma(e_{-1}, X_2) = -1$
and it is disjoint from the image of $G_1$ and $X_1$.
Finally, remove the arc $e_2$ and add a new arc $e_{-2}$ from $v_w$ to $v_b$, so that $\sigma(e_2, X_2) = -1$ and the rotation at $v_b$
becomes
\[
v_b: \dots e_0 e_{-1} e_{-2} e_{2'} e_3,
\]
see Figure~\ref{fig.modification_4k}(b). Let $C = \{X_1, X_2\}$. 
%{\color{orange}
 We have
$\sigma(e_{-1}, C) = -1$, $\sigma(e_{-2}, C) = -2$ and $\sigma(e_{2'}, C) = 2$.
%} % \m{susikirtimų skaičiaus invariantiškumas\dots}
By the earlier argument, $\sigma(e_k, C) = \sigma(e_k, \{X_1\}) =k$ for all $k \in \{0, \dots, n-3\} \setminus \{2\}$.
Thus for the new embedding $\mathcal{M}_2$ of $D_n$ with arcs $e_{-2}, e_{-1}, e_0, e_1, e_2', e_3, \dots, e_{n-3}$
the voltage function $\alpha_{\mathcal{M}_2, C, n}: D_n \to \mathbb{Z}_n$ is bijective and the proof follows by Proposition~\ref{prop.cuts}.

\end{proof}

\subsection{Completing the proof for 3-dimensional embeddings}

An embedding as in Theorem~\ref{thm.geometric.opt} has a face invariant under the $n$-fold rotational symmetry, and hence at least one point in that face is a fixed point under the rotation. Thus Theorem~\ref{thm.geometric.opt} follows from Proposition~\ref{prop.rotation} and the next result.
\begin{lemma} \label{lem.geometric.opt}
    Let $n \ge 2$ be an integer. 
    Let $\mathcal{M}$ be a piecewise linear embedding of $K_{n,n}$ into a closed connected orientable surface $\S \subset \mathbb{R}^3$. 
    Suppose $\mathcal{M}$ has n-fold rotational symmetry such that no vertex of the graph is mapped to a fixed point
    but 
    \begin{itemize}
        \item[($\bullet$)] some vertex of $\S$ is a fixed point.
    \end{itemize}
    %If $\mathcal{M}$ is 2-cell
    Then the genus of $\S$ is at least $\Lgeom(n)$.
    
    Furthermore, if $n \ne 4$, this lower bound is best possible and there exists a genus $\Lgeom(n)$ embedding satisfying
    the assumption of Theorem~\ref{thm.geometric.opt}.
\end{lemma}

\begin{proof} %{Lemma~\ref{lem.geometric.opt}}
   The lemma is trivial for $n=2$. %The case $n=4$ is treated in Lemma~\ref{lem.geometric.4}.
   
   %%First assume $\S$ contains a fixed point under the $n$-fold rotation (note that by the fixed point theorem, every embedding as in Theorem~\ref{thm.geometric.opt}, has such a point). 
   
   Suppose $n \ge 3$ and the embedding $\mathcal{M}$ of a complete bipartite graph $G$ into $\S$ satisfies all the assumptions of the statement. %By Proposition~\ref{prop.rotation}, no vertex of $G$ is mapped to a fixed point under the rotation.
    %Therefore we can apply 
    Use Proposition~\ref{prop.cuts} to obtain an embedding $\mathcal{M}_B$ of a base graph $G_B$ in a base surface $\S_B$ together with a set of curves $C$ in $\S_B$, such that $G$ is isomorphic to the graph derived from the voltage graph $(G_B, \alpha_{\mathcal{M}_B, C, n})$. Since $G$ is bipartite it follows that $|V(G_B)| = 2$, $|E(G_B)| = n$, and no edge is a loop. By inverting voltages if necessary, we can assume each edge of $G_B$ is oriented from one vertex, call it $v_w$, to another vertex, call it $v_b$. Next, since $G$ is derived from $G_B$, it must be that the image $\{\alpha_{\mathcal{M}_B, C, n}(e), e \in E(G_B)\}$ contains every element of $\mathbb{Z}_n$.
   By Proposition~\ref{prop.cuts} the genus of $\S$ is $ng + (n-1)  (t-1)$, where $g$ is the genus of $\S_B$ and $t$, $t \ge 1$, is the number of arcs in $C$. Now Lemma~\ref{lem.geometric_lower} implies that the genus of $\S$ is at least $\Lgeom(n)$.

Finally, the upper bound, or the fact that for each $n \ge 3$, $n \ne 4$ there is a symmetric embedding of genus $\Lgeom(n)$ follows by Lemma~\ref{lem.geomupper}.

% an $n$-fold symmetric embedding of $K_{4,4}$ into a surface of genus $\frac {n^2} 4 = 4$ satisfying the conditions of the the theorem statement corresponds to the embedding from the Pinavia interchange \cite{jbk2010}. {\color{orange} The proof that this genus is best possible
%is given in Lemma~\dots.}
\end{proof}

It remains to consider the special case $n=4$.

\begin{figure}[t]
    \centering
    \includegraphics[width=0.3\linewidth]{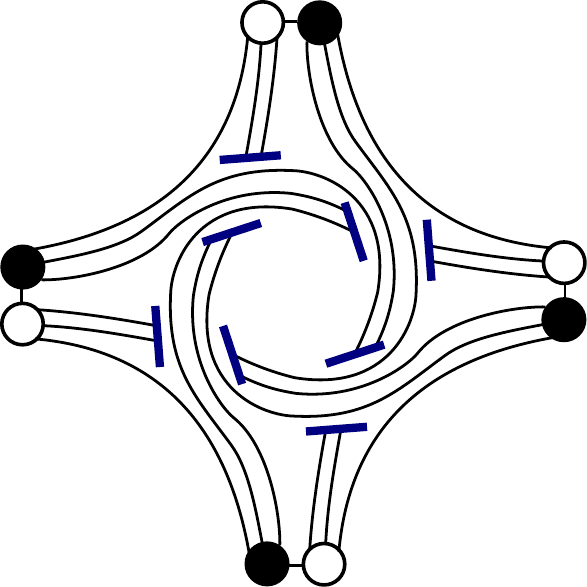}
    \caption{$n=4$: the ``Pinavia'' interchange \cite{jbk2010}. }
    \label{fig.pinavia_embedding}
\end{figure}

\begin{lemma}\label{lem.geometric.4}
    In the case $n=4$ of both Lemma~\ref{lem.geometric.opt} and Theorem~\ref{thm.geometric.opt} the minimum genus is $4$.
\end{lemma}
\begin{proof}
   %Let $\mathcal{M}': G' \to  \mathbb{S}'$ be the embedding of $K_{4,4}$ of genus $g'$ that satisfies the conditions of the theorem and has the minimal genus.
   %By Proposition~\ref{prop.rotation}, no vertex of $G'$ is mapped to a fixed point under the rotation.
   %Apply Proposition~\ref{prop.cuts} to obtain $C$, $t$, $g$ and $\mathcal{M}: G \to \mathbb{S}$, where $G$ is a copy of $D_n$.
   Suppose the minimum genus is $g' < 4$.
   Consider the genus $g$ embedding $\mathcal{M}$ of $D_4$ and the set $C$ with $t$ cut arcs obtained from a genus $g'$ counterexample using Proposition~\ref{prop.rotation} and Proposition~\ref{prop.cuts}. 
   
   %Assume $g' < 4$. 
   By Lemma~\ref{lem.geometric_lower}, we must have $g=0$ and $t=2$. That is, we have a $D_4$, two cut arcs and some number of cut loops embedded into the sphere, such that each arc of the embedded graph makes a different number (modulo 4) of
   intersections with the cut arcs, as defined by (\ref{eq.cutvolt}).

   Let the rotation at the white vertex be $v_w: e_0 e_1 e_2 e_3$. By Euler's formula, the set of regions of $\mathcal{M}$ consists of four 2-faces,
   let $F_i$ be bounded by $e_i$ and $e_{(i+1) \bmod 4}$.
   Write $\sigma_X = (\sigma(e_0, X), \sigma(e_1, X), \sigma(e_2,X), \sigma(e_3,X))$.
    Using the argument of the proof of Lemma~\ref{lem.geometric_lower},
   each of the four endpoints of the cut arcs must be contained in a different face. By the same argument,
   the contributions of a single cut arc $c$ are $\sigma_c=(x, x+s, x, x)$ if the endpoints
   of $c$ are in $F_0$ and $F_1$, and $\sigma_c=(x, x+s, x+s, x)$ if the endpoints of $c$ are in $F_0$ and $F_2$. Here
   $x$ is some integer, $s \in \{1, -1\}$ and we considered two possible ways to pick the first edge for an intersection from each endpoint and two ways to direct $C$. Similarly, for a cut loop $c_1 \in C$, $\sigma_{c_1}=(z_1, z_1, z_1, z_1)$ for some integer $z_1$. 
   
   Summing all contributions, the image of $\sigma_C$ is either $\{y, y+s_1+s_2\}$, or $\{y, y+s_1, y+s_2\}$ or $\{y, y+s_1, y+s_1+s_2, y+s_2\}$, for some $y \in \mathbb{Z}$ and $s_1, s_2 \in \{-1, 1\}$. Therefore the voltage function $\alpha_{\mathcal{M}, C, 4}$ is not bijective, a contradiction. 

   Finally, one of the solutions of genus 4 is the Pinavia interchange, see Figure~\ref{fig.pinavia_embedding}. Another solution
   could be obtained by using the construction of Lemma~\ref{lem.geomupper}.
   \end{proof}

\subsection{Other observations}

One may ask what happens if we drop the assumption that $\S$ contains a fixed point.

\begin{lemma} \label{lem.geometric.opt2}
    Let $n, \mathcal{M}, \S$ be as in Lemma~\ref{lem.geometric.opt}, but drop the assumption ($\bullet$).
    Then the genus of $\S$ is at least 
    \begin{align*}
    \Lgeomt(n)=
    \begin{cases}
        \frac {(n-2)^2} 4, &\mbox{if }  n\equiv0\,(\bmod \, 4);  \\
        \frac {n(n-1)} 4, &\mbox{if }  n\equiv1\,(\bmod \, 4); \\
        \frac {n (n-2)} 4, &\mbox{if }  n\equiv2\,(\bmod \, 4);  \\
        \frac {n^2 - 3n + 4} 4, &\mbox{if }  n\equiv3\,(\bmod \, 4),
    \end{cases}
    \end{align*}
    and this bound is best possible.
\end{lemma}

\begin{proof}
Suppose $\S$ has no fixed points. 
As remarked in the comment after Proposition~\ref{prop.cuts}, the implication (a)$\implies$(b) still holds yielding an embedding $\mathcal{M}_B$ of $D_n$ into a surface $\S_B$ as in the proof of Lemma~\ref{lem.geometric.opt}, but now the family $C$ of curves contains $t=0$ arcs, so all curves in $C$ are loops. Using the same argument as in Lemma~\ref{lem.geometric_lower}, $\mathcal{M}_B$ can have no regions with boundary of size 2. Writing $r$ for the number of regions and $g$ for the genus of $\S_B$, we have by Euler's formula
\begin{align*}
g \ge \left \lceil \frac{2 - 2 - r + n} 2 \right \rceil \ge \left \lceil \frac{n - \lfloor \frac {2n} 4 \rfloor} 2 \right \rceil = \left \lceil \frac n 4 \right \rceil.
\end{align*}
So by the Riemann-Hurwitz formula the genus of $\mathcal{M}$ is 
\begin{align*}
n g + (n-1)(t-1) = n g - (n-1) \ge  \Lgeomt(n).
%                       \begin{cases}
%                         \frac{ (n-2)^2} 4, &\mbox{if $n$ is even}; \\
%                         \frac{n^2 - 3n + 4} 4, &\mbox{if $n$ is odd}. \\
%                         \end{cases}
\end{align*}
Note that for $n \bmod 4 \in\{1,2\}$ we have $\Lgeomt(n) = \Lgeom(n)$ and the proof follows by Lemma~\ref{lem.geometric.opt} (for such $n$ the last inequality is strict, so there are no minimal genus constructions without fixed points in $\S$).

For $n \equiv 0 \,(\bmod\,4)$, $n \ge 8$ we construct an embedding $\mathcal{M}_B$ of $D_n$ into a torus with a set $C$ consisting of a single cut loop and genus $n/4$ as follows (see also Figure~\ref{fig.nonfix}~(a)). Start with the construction $\B(n-2)$ of the proof of Lemma~\ref{lem.geomupper}, and let $\pi^{(n-2)} = (\pi_1, \dots, \pi_{n-3})$ and $\pi'^{(n-2)} = (\pi_1', \dots, \pi_{n-3}')$ be the respective permutations. Extend these permutations by defining $\tilde{\pi} = (\pi_1, \dots, \pi_{n-3}, n-2, n-1)$ and $\tilde{\pi}' = (\pi_1', \dots, \pi_{n-3}', n-3, n-2)$ and add points $n-3, n-2$ on the segment $a b$ and the points $n-2, n-1$ on the segment $d c$ of $\B(n-2)$. Now identify $da$ with $cb$ in $\B(n-2)$ to get a cylinder with $(n-4)/4$ handles. Add edges $\tilde{\pi}_{n-2}' \tilde{\pi}_{n-3}$, $\tilde{\pi}'_{n-1} \tilde{\pi}_{n-2}$, $v_w \tilde{\pi}_{n-2}$, $v_w \tilde{\pi}_{n-1}$ and $v_w v_b$ so that the rotations are
\begin{align*}
& \tilde{\pi}'_{n-2}: \tilde{\pi}_{n-3} && \tilde{\pi}'_{n-1}: \tilde{\pi}_{n-2} && \tilde{\pi}_{n-3}: v_w \tilde{\pi}_{n-2}' \\
& \tilde{\pi}_{n-2}: \tilde{\pi}'_{n-1} v_w && \tilde{\pi}_{n-1}: v_w   && v_w: v_b \dots \tilde{\pi}_1 \tilde{\pi}_{n-1} \tilde{\pi}_{n-2}.
\end{align*}
Now let $\mathcal{M}'$ be the embedding obtained by expanding the vertices $\tilde{\pi}_i, \tilde{\pi}'_i$, $i \in \{1, \dots, (n-1)\}$ similarly as in the proof of Lemma~\ref{lem.geomupper}. If we identify the edges $ab$ and $dc$, and let $C$ be the cut loop consisting of the single closed curve $\tilde{c}$ that coincides with $dc$, we get a base embedding $\mathcal{M}_B$ of genus $(n-4)/4 + 1 = n/4$ as required by the lower bound. We can see that $(b) \implies (a)$ from Proposition~\ref{prop.cuts} holds in this particular case even with $t=0$: construct $n$ copies of $\mathcal{M}'$ and glue them symmetrically into a ``torus with handles''. The resulting surface is connected and has genus of $n(n-4)/4 + 1 = \Lgeomt(n)$. The embedding obtained by gluing the copies of $\tilde{\pi}_i'$ and $\tilde{\pi}_i$ and smoothing the vertices of degree 2 is that of $K_{n,n}$ by the same argument as in the proof of Lemma~\ref{lem.geomupper}.

For $n=4$ take $\tilde{\pi}=(1,2,3)$ and the rotations $v_w: (v_b, \tilde{\pi}_1, \tilde{\pi}_3,\tilde{\pi_2})$,  $\tilde{\pi}_1: v_w \tilde{\pi}_2'$, $\tilde{\pi}_2: \tilde{\pi}_3' v_w$, $\tilde{\pi}_3: v_w$, $\tilde{\pi}_1': v_b$, $\tilde{\pi}_2': \tilde{\pi}_2$ and $\tilde{\pi}_3': \tilde{\pi}_3$. There is also an interesting genus 1 construction with rotational symmetry group $C_{2n}=C_8$ (not colour-preserving) and genus 1, see Figure~\ref{fig.nonfix}~(b-c).

\begin{figure}
\centering
\centering
\begin{subfigure}[b]{0.3\textwidth}
\centering
\includegraphics[width=3cm]{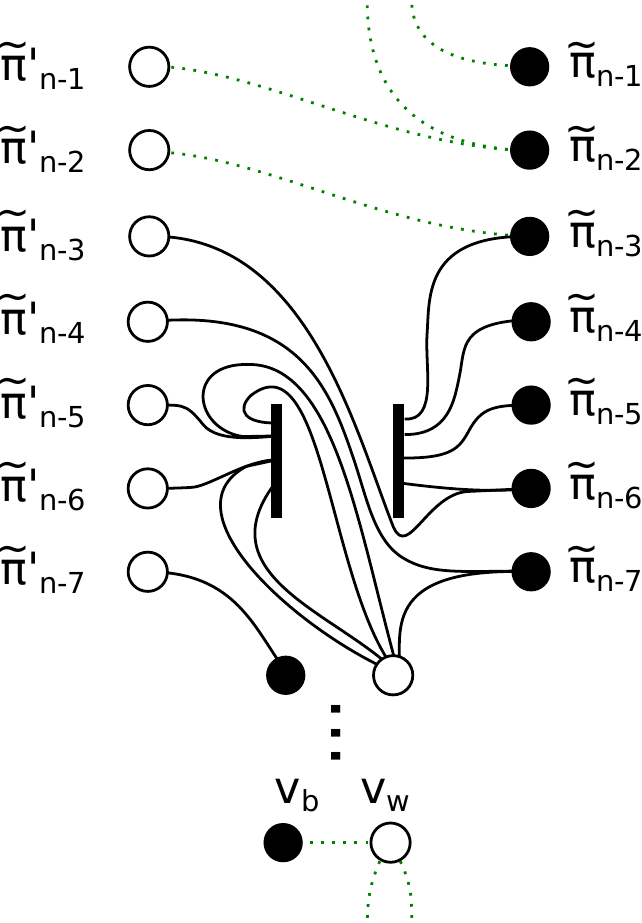}
\caption{}
\end{subfigure}
\begin{subfigure}[b]{0.3\textwidth}
\centering
\includegraphics[height=2cm]{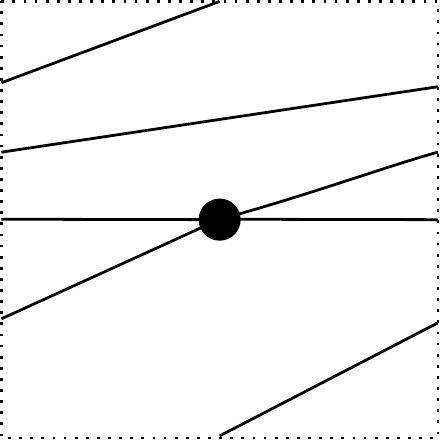}
\caption{}
\end{subfigure}
\begin{subfigure}[b]{0.3\textwidth}
\centering
\includegraphics[height=2cm]{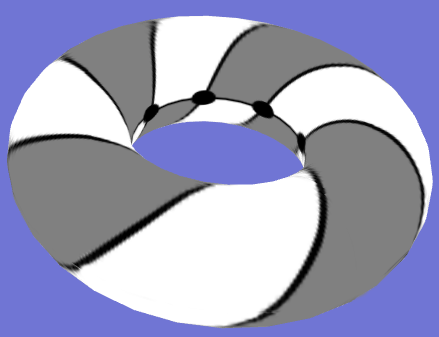}
\caption{}
\end{subfigure}
\caption{(a) The green dotted edges are added in the construction of $\mathcal{M}'$ for $n\equiv 0\,(\bmod\,4)$; (b) a genus 1 base embedding that gives a solution (c) for $n=4$ with symmetry $C_8$.}
\label{fig.nonfix}
\end{figure}

Finally, the construction for $n=3\,(\bmod\,4)$ of genus $\Lgeomt(n)$ is obtained in a similar way as the construction for $n=1\,(\bmod\,4)$, except that we do not need to add the points $\tilde{\pi}_{n-2}$ and $\tilde{\pi}_{n-2}'$.
\end{proof}
%
%\medskip
%
%The optimal ringroad constructions obtained in the proofs of Theorem~\ref{thm.geometric.opt} and Lemma~\ref{lem.geometric.opt2} have some useful properties. For example, we can construct quadrangular embeddings with $n$-fold rotational symmetry for $K_{n,k}$ with any $k \equiv 0\,(\bmod\,4)$ and $n \ge k$.

\appendix

\newpage

\section{Symmetric solutions for \texorpdfstring{$5 \le n \le 9$}{small n}}

Above we presented complete symmetric interchanges of minimum genus for $n \in \{3,4\}$ and simple procedure to construct rotationally symmetric interchanges with the minimum number of bridges for arbitrary $n \ge 5$. Below we present complete examples obtained as in the proof of Lemma~\ref{lem.geomupper} for $n$ up to $9$.  Note, that we omitted the edges $v_w^{(i)} v_b^{(i)}$ $i \in \mathbb{Z}_n$, which are easy to add and which correspond to lanes for turning around. Also we show the vertices $\pi_{k}^{(i)} = \pi_{k}'^{(i+1)}$, $k \in \{1, \dots, n\}$, see the proof of Lemma~\ref{lem.geomupper}.

\bigskip
%
%\begin{figure}[h]
%    \centering
%    \includegraphics[width=0.5\linewidth]{Keturiu_krypciu.jpg}
%    \caption{$n=4$: ``Pinavia'' interchange \cite{jbk2010}. }
%\end{figure}
%

\begin{figure}[h]
    \centering
    \includegraphics[width=0.6\linewidth]{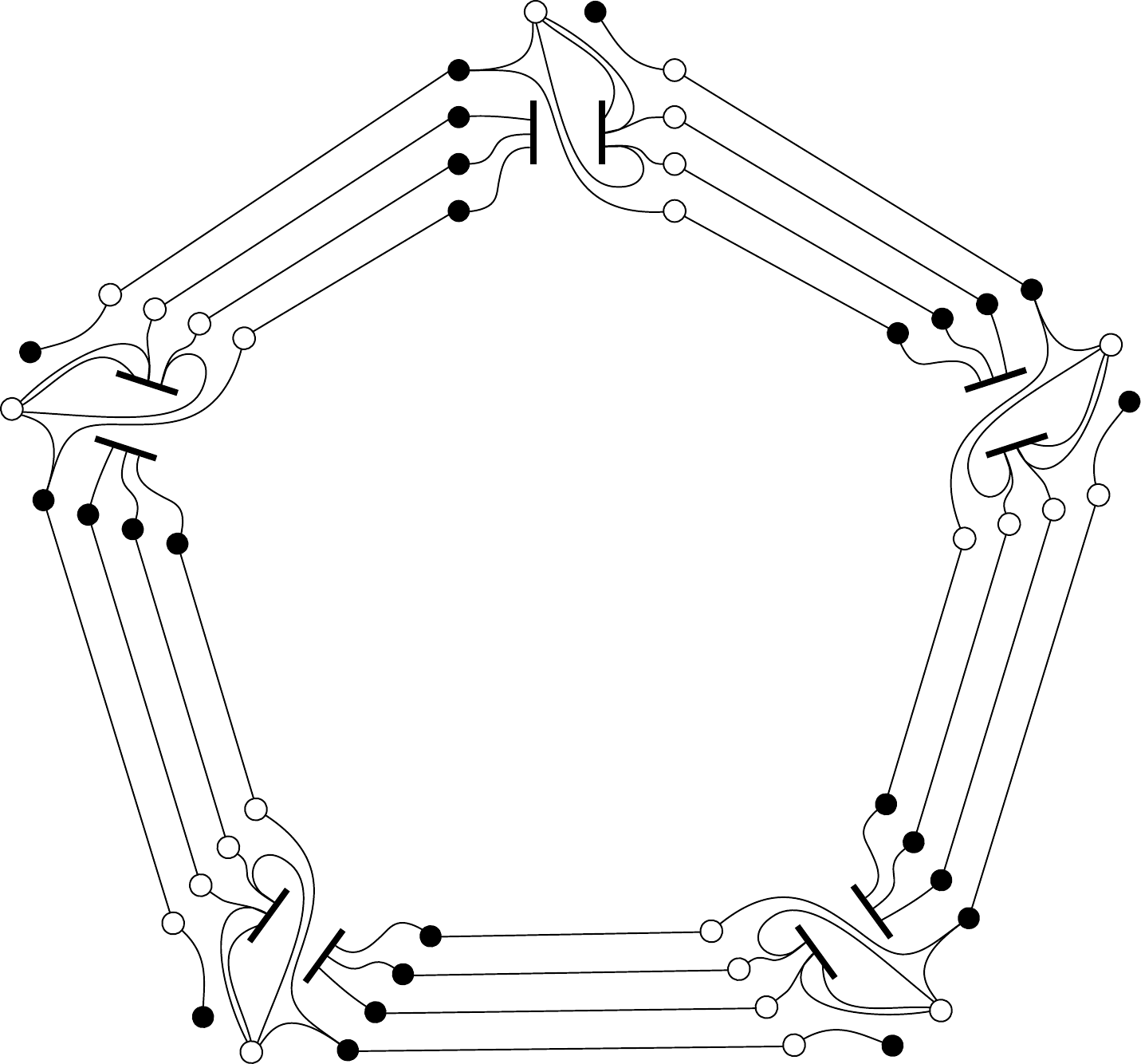}
    \caption{$n=5$.}
\end{figure}
\begin{figure}
    \centering
    \includegraphics[width=0.6\linewidth]{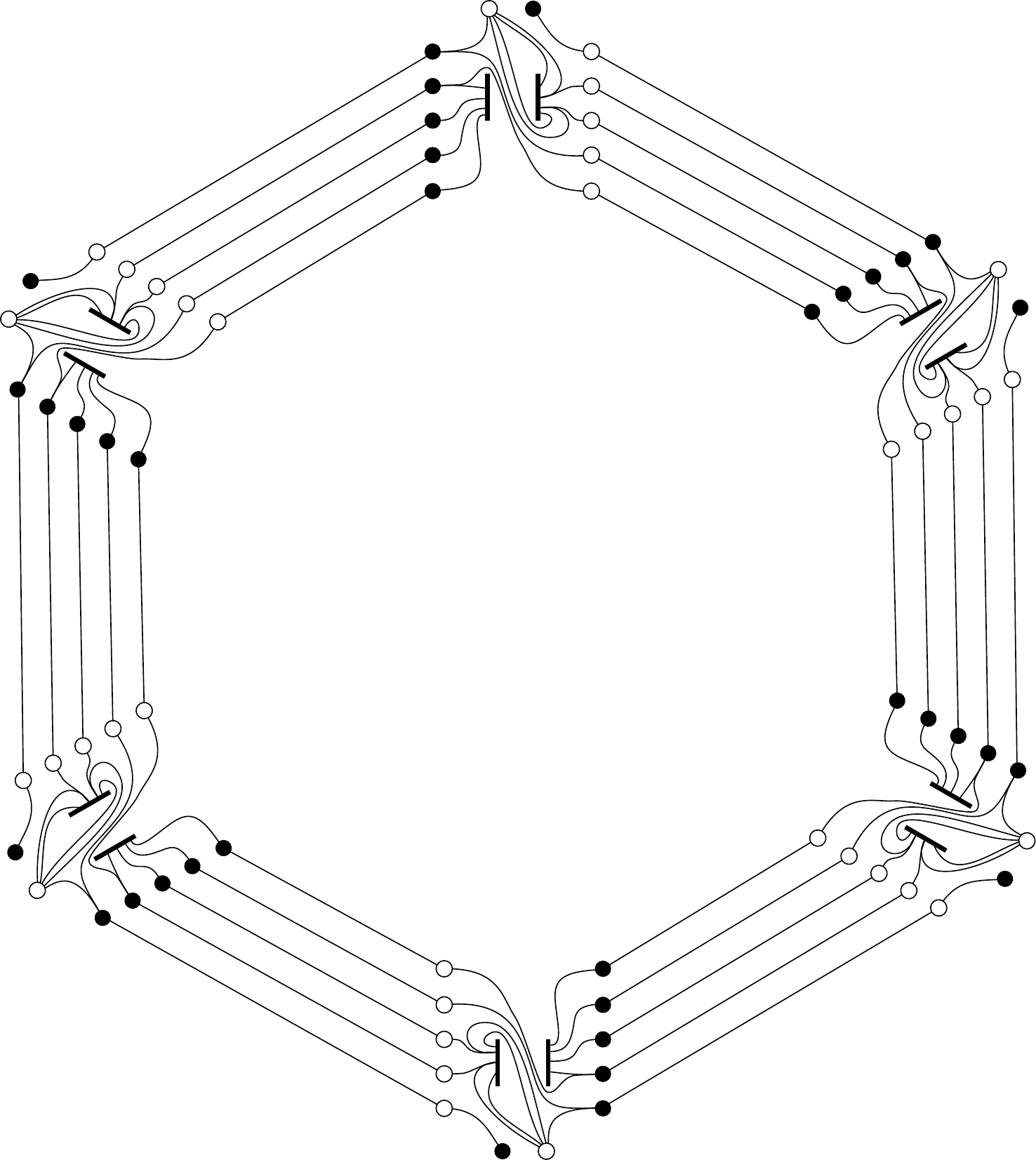}
    \caption{$n=6$.}
\end{figure}

\begin{figure}
    \centering
    \includegraphics[width=0.6\linewidth]{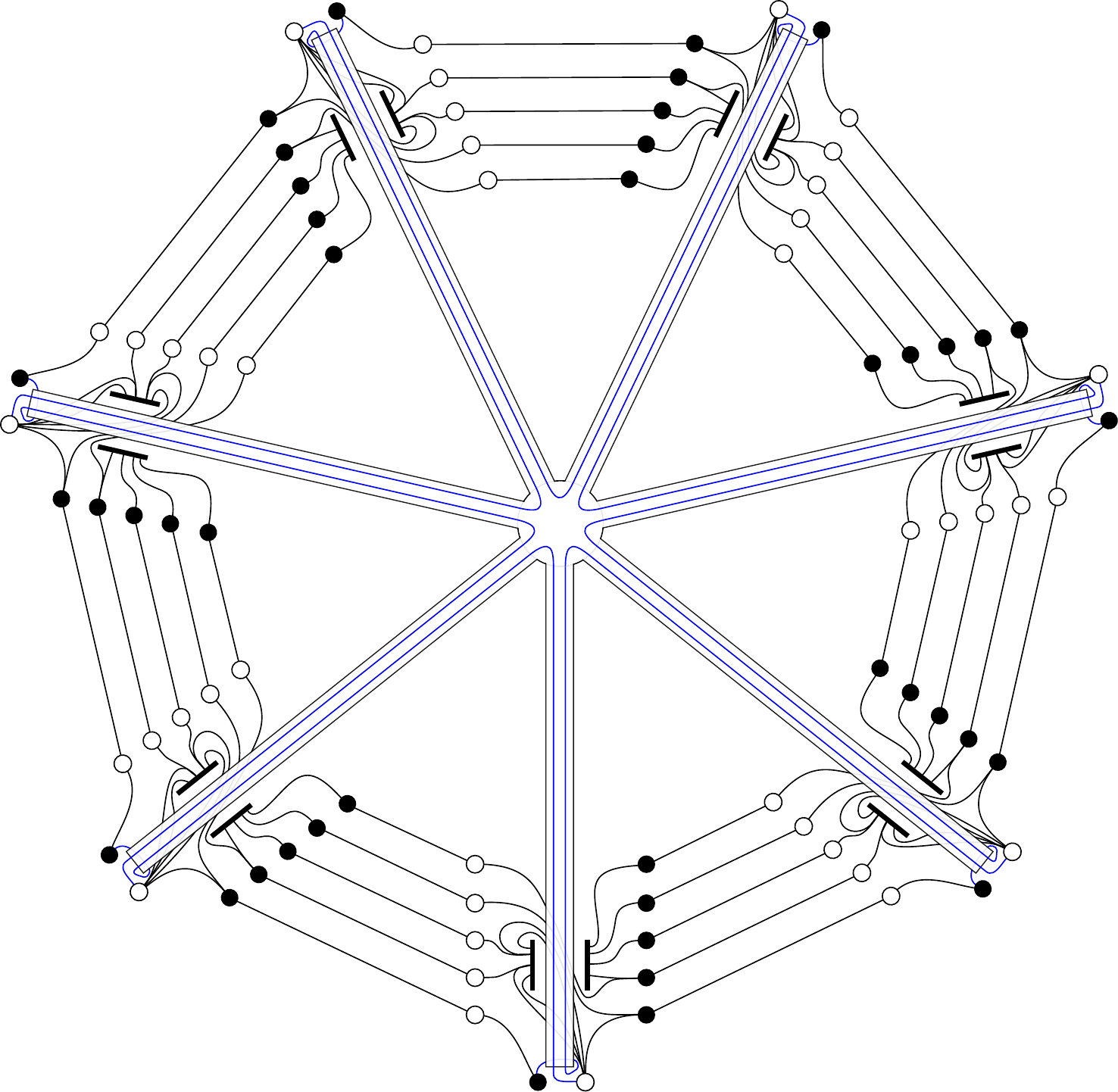}
    \caption{$n=7$. The orbit of $e_{-1}$ is shown in blue. }
\end{figure}

\begin{figure}
    \centering
    \includegraphics[width=0.6\linewidth]{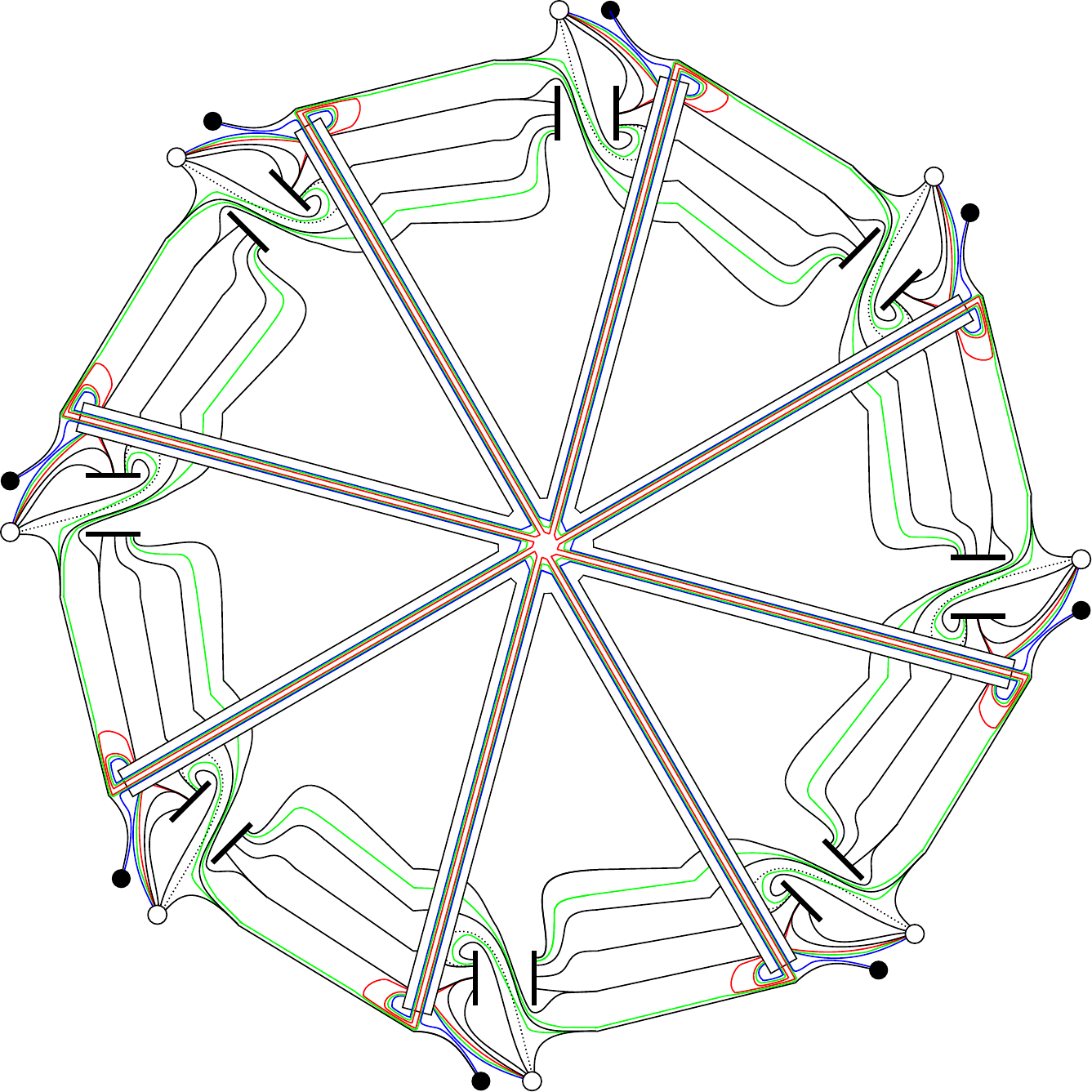}
    \caption{$n=8$. The orbits of $e_{-1}$, $e_{-2}$ and $e_2'$ are shown in blue, green and red respectively, and the orbit of the removed edge $e_2$ shown dotted.}
\end{figure}

\begin{figure}
    \centering
    \includegraphics[width=0.6\linewidth]{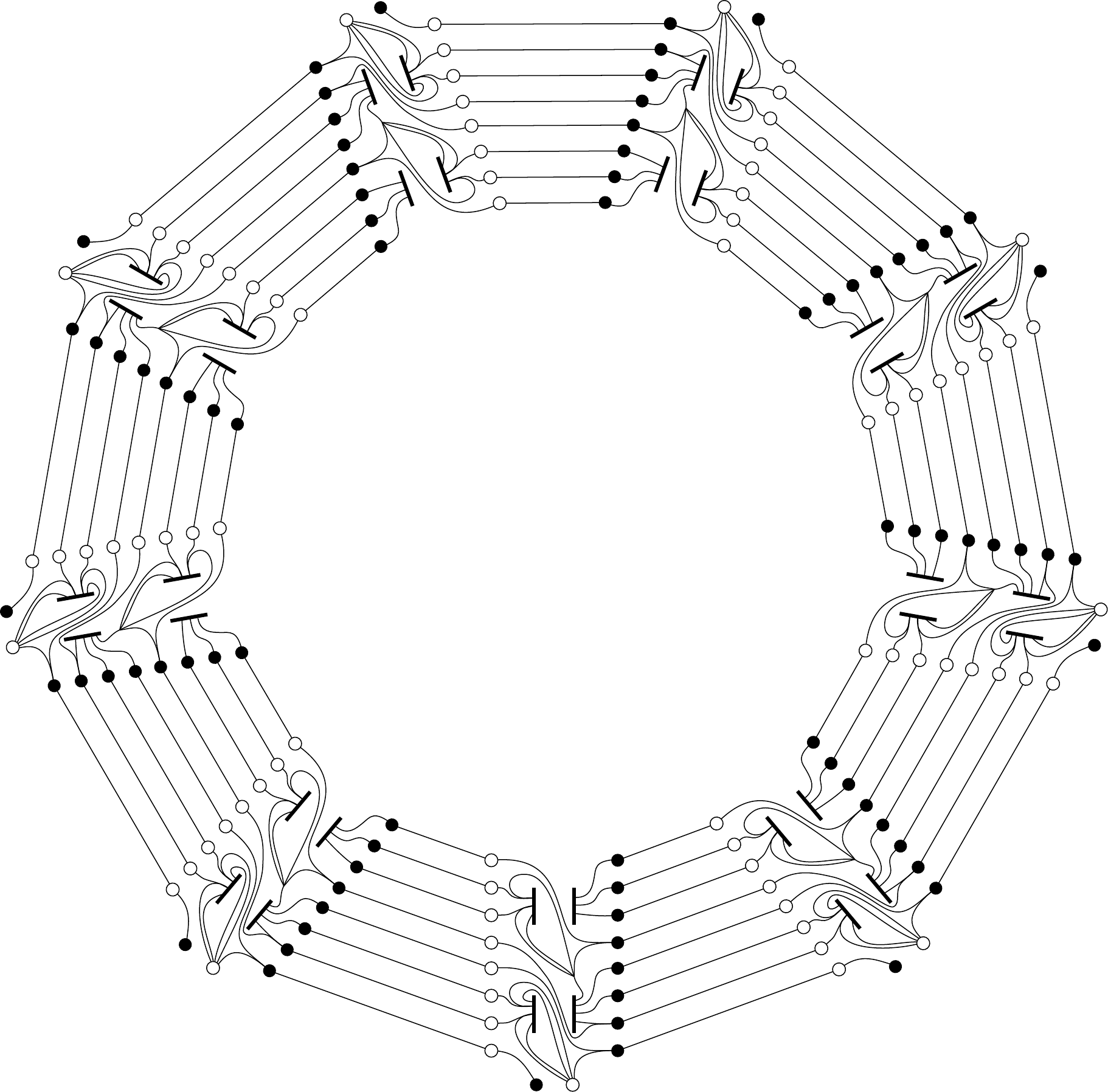}
    \caption{$n=9$.}
\end{figure}

\end{document}